\documentclass[a4paper,12pt,english]{amsart}
\usepackage{amsmath,latexsym,amssymb,amsfonts}
\usepackage{amscd,amssymb,epsfig}
\usepackage[arrow, matrix, curve]{xy}
\usepackage{hyperref}
\usepackage{mathdots}
\usepackage{enumerate}
\usepackage{tikz}
\usetikzlibrary{decorations.markings}
\tikzset{->-/.style={decoration={markings,mark=at position #1 with {\arrow{>}}},postaction={decorate}}}

\numberwithin{equation}{section}

\newtheorem{theorem}{Theorem}[section]
\newtheorem{lemma}[theorem]{Lemma}
\newtheorem{corollary}[theorem]{Corollary}
\newtheorem{proposition}[theorem]{Proposition}

\theoremstyle{definition}
\newtheorem{definition}[theorem]{Definition}

\newtheorem{example}[theorem]{Example}
\newtheorem{remark}[theorem]{Remark}

\renewcommand\labelenumi{(\roman{enumi})}
\renewcommand\theenumi\labelenumi

\renewcommand{\epsilon}{\varepsilon}

\newcommand{\A}{\mathcal{A}}
\newcommand{\B}{\mathcal{B}}
\newcommand{\C}{\mathbb{C}}
\newcommand{\D}{\mathcal{D}}

\newcommand{\F}{\mathcal{F}}
\newcommand{\FF}{\mathbb{F}}
\renewcommand{\H}{\mathcal{H}}
\newcommand{\cL}{\mathcal{L}}
\newcommand{\M}{\mathcal{M}}
\newcommand{\N}{\mathbb{N}}

\newcommand{\R}{\mathbb{R}}

\newcommand{\bS}{\mathbb{S}}

\newcommand{\X}{\mathbb{X}}

\newcommand{\ev}{\operatorname{ev}}
\newcommand{\Ev}{\operatorname{Ev}}
\newcommand{\tr}{\operatorname{tr}}

\newcommand{\vN}{\operatorname{vN}}

\newcommand{\Tr}{\operatorname{Tr}}
\newcommand{\im}{\operatorname{im}}

\newcommand{\dom}{\operatorname{dom}}

\newcommand{\rank}{\operatorname{rank}}

\renewcommand{\1}{\mathbf{1}}
\newcommand{\0}{\mathbf{0}}

\newcommand{\ootimes}{\mathbin{\overline{\otimes}}}

\newcommand{\full}{\mathrm{full}}

\renewcommand{\star}{\bigstar}

\makeatletter
\def\moverlay{\mathpalette\mov@rlay}
\def\mov@rlay#1#2{\leavevmode\vtop{%
\baselineskip\z@skip \lineskiplimit-\maxdimen
\ialign{\hfil$#1##$\hfil\cr#2\crcr}}}
\makeatother

\def\plangle{\moverlay{(\cr<}}
\def\prangle{\moverlay{)\cr>}}

\makeindex

\title[The free field]{The free field: realization via unbounded operators and Atiyah property}

\author[T. Mai]{Tobias Mai}
\address{Saarland University, Faculty of Mathematics, D-66123 Saarbr\"ucken, Germany}
\email{mai@math.uni-sb.de}

\author[R. Speicher]{Roland Speicher}
\address{Saarland University, Faculty of Mathematics, D-66123 Saarbr\"ucken, Germany}
\email{speicher@math.uni-sb.de}

\author[S. Yin]{Sheng Yin}
\address{Saarland University, Faculty of Mathematics, D-66123 Saarbr\"ucken, Germany}
\email{yin@math.uni-sb.de}

\date{\today}

\thanks{We thank Ken Dykema for discussions about the relation between the Atiyah property and the rational closure, as well as providing us with the example in \ref{ex:Dykema_Pascoe}. We also thank Christophe Reutenauer and Konrad Schrempf for discussions about the free field and providing relevant literature.\\
This work has been supported by the ERC Advanced Grant NCDFP 339760 and by the SFB-TRR 195, Project I.12.}



\begin{document}

\begin{abstract}

Let $X_1,\dots,X_n$ be operators in a finite von Neumann algebra and consider their division closure in the affiliated unbounded operators. 
We address the question when this division closure is a skew field (aka division ring) and when it is the free skew field. We show that the first property is equivalent to the strong Atiyah property and that the second property can be characterized in terms of the non-commutative distribution of $X_1,\dots,X_n$. More precisely, $X_1,\dots,X_n$ generate the free skew field if and only if there exist no non-zero finite rank operators $T_1,\dots,T_n$ such that $\sum_i[T_i,X_i]=0$. Sufficient conditions for this are the maximality of the free entropy dimension or the existence of a dual system of $X_1,\dots,X_n$. Our general theory is not restricted to selfadjoint operators and thus does also include and recover the result of Linnell that the generators of the free group give the free skew field.

We give also consequences of our result for the question of atoms in the distribution of rational functions in free variables or in the asymptotic eigenvalue distribution of matrices over polynomials in asymptotically free random matrices.


\end{abstract}

\maketitle

\tableofcontents

\section{Introduction}

In the last years there has been quite some progress on understanding qualitative and quantitative properties of 
the limit of multivariate random matrix models on one side and of
the non-commutative distribution of several (in general, non-commuting) operators on infinite-dimensional Hilbert spaces on the other side.
Those two occurrences of non-commutative distributions are closely related; namely, free probability theory has taught us that the limit of random matrix ensembles is, in many situations, given by operators in interesting operator algebras. Hence random matrices can tell us something about operator algebras and operator theory provides tools for dealing with asymptotic properties of random matrices.

There are two prominent limits of random matrix models: the limit of independent GUE matrices is given by free semicircular variables $S_1,\dots,S_n$; and the limit of independent Haar unitary random matrices is given by free Haar unitaries $U_1,\dots,U_n$. 
Those constitute benchmarks of non-commutative distributions to which we would like to compare other situations, given by general operators $X_1,\dots,X_n$ (which might be limits of other random matrix models).
We know that the von Neumann algebra generated by $S_1,\dots,S_n$ and by $U_1,\dots,U_n$ is in both cases the free group factor $L(\FF_n)$. Results about von Neumann algebras for general $(X_1,\dots,X_n)$ are usually quite hard to achieve. We will here not address such von Neumann algebra questions, but instead change the perspective to a more algebraic one.

Consider bounded operators $X_1,\dots,X_n$ acting on a Hilbert space $\H$. Let us denote by $\C\langle X_1,\dots,X_n\rangle\subset B(\H)$ the algebra generated by those operators. Then we know that in the case of the $S_i$ or the $U_i$ we do not have any non-trivial algebraic relations between the $S_i$ or between the $U_i$, and hence in those cases the generated algebras are just the free algebra of non-commutative polynomials $\C\langle x_1,\dots,x_n\rangle$ in non-commuting indeterminates $x_1,\dots,x_n$. The same is true for many other operators $X_1,\dots,X_n$ -- and that's the situation we will restrict to in the following. Instead of taking now an analytic closure of $\C\langle X_1,\dots,X_n\rangle$ in the strong operator topology within the bounded operators, and thus getting a von Neumann algebra $\M:=\vN(X_1,\dots,X_n)\subset B(\H)$, we will push the algebraic setting to its limit, by taking an algebraic closure, also allowing inverses. For this we will consider the division closure of $\C\langle X_1,\dots,X_n\rangle$ within the unbounded operators affiliated to $\M$. One should note that we are always working in the setting of finite von Neumann algebras, which means that the affiliated unbounded operators have all the nice properties one knows from the finite-dimensional setting. In particular, the set $\A$ of such unbounded affiliated operators forms a $*$-algebra and $A\in\A$ is invertible (as an unbounded operator) if and only if its kernel is trivial. We will denote the division closure of $\C\langle X_1,\dots,X_n\rangle$ by
$\C\plangle X_1,\dots,X_n\prangle$; this is, by definition, the smallest subalgebra of $\A$, which contains $\C\langle X_1,\dots,X_n\rangle$ and is closed under taking inverses, whenever possible in $\A$.

The only situation where such a division closure has been determined so far is the case of free Haar unitaries. Those are the generators of the free group and  Linnell showed in \cite{Lin93} that their division closure is the so-called ``free skew field'' (or just ``free field'')
$\C\plangle x_1,\dots,x_n\prangle$. The latter is a purely algebraic object, also known under the name of ``non-commutative rational functions''. It is a skew field (aka division ring) and 
is given by all meaningful rational expressions in the non-commuting formal variables $x_1,\dots,x_n$, where two such expressions are being identified if they can be transformed into each other by algebraic manipulations. The existence of such an object is quite non-trivial, and was established by Amitsur \cite{Ami66} and extensively studied by Cohn \cite{Coh95,Coh06}, Cohn and Reutenauer \cite{CR94,CR99}, and also, more recently, in \cite{KV12,LR13,GGOW16,Vol18,Sch18}.

We will show as one of our main results that for many operator tuples -- in particular, for free semicirculars --  we will get as their division closure the free field. Actually, we will give precise characterizations for when this will happen. Here is our main theorem on this. After the theorem we will explain the used notations.

\begin{theorem}\label{thm:intro-main}
Let $(\M,\tau)$ be a tracial $W^{\ast}$-probability space (i.e., $\M$ is a von Neumann algebra and $\tau:\M\to\C$ is a faithful normal trace on $\M$) and $\A$ the $\ast$-algebra of unbounded operators affiliated to $\M$. Consider a 
tuple $X=(X_{1},\dots,X_{n})$ of (not necessarily selfadjoint) operators in $\M$.
Then the following statements are equivalent.
\begin{enumerate}
\item
The evaluation map $$\ev_{X}:\C\left\langle x_{1},\dots,x_{n}\right\rangle \rightarrow\A$$ -- which is the unital homomorphism sending $x_i$ to $X_i$ -- extends to an injective homomorphism 
$$\Ev_{X}: \C\plangle x_{1},\dots,x_{n}\prangle\rightarrow\A,$$ 
whose image is the division closure $\C\plangle X_1,\dots,X_n\prangle$; in this sense, the division closure $\C\plangle X_1,\dots,X_n\prangle$ ``is'' the free field $\C\plangle x_{1},\dots,x_{n}\prangle$.
\item 
For any $N\in\N$ and $P\in M_{N}(\C\left\langle x_{1},\dots,x_{n}\right\rangle )$ we have: if $P$ is linear and full, then $P(X)\in M_{N}(\A)$ is invertible.
\item 
For any $N\in\N$ and $P\in M_{N}(\C\left\langle x_{1},\dots,x_{n}\right\rangle )$ we have: if $P$ is full, then $P(X)\in M_{N}(\A)$ is invertible.
\item 
For any $N\in\N$ and $P\in M_{N}(\C\left\langle x_{1},\dots,x_{n}\right\rangle )$ we have: $\rank(P(X))=\rho(P)$.
\item 
$\Delta(X_1,\dots,X_n)=n$.

\end{enumerate}
\end{theorem}

The second, third, and fourth characterization address matrices over polynomials in our operators; in (ii) we consider linear matrices, where the entries have degree at most 1. Again we compare the formal algebraic object $P\in M_{N}(\C\left\langle x_{1},\dots,x_{n}\right\rangle)$ with the corresponding matrix $P(X)\in M_N(\M)$ where we replace each $x_i$ by $X_i$. (ii) and (iii) require that the operator $P(X)$ is invertible, as an unbounded operator in $M_N(\A)$, if and only if $P$ is invertible as a matrix over the free field. The latter invertibility can be characterized in purely algebraic terms, by the fact that $P$ is full. This means that $P$ cannot be factorized into a product of strictly rectangular matrices over $\C\left\langle x_{1},\dots,x_{n}\right\rangle$. Also in the case that $P\in M_{N}(\C\left\langle x_{1},\dots,x_{n}\right\rangle)$ is not full, one can relate the algebraic with the analytic side, as in (iv). The inner rank $\rho(P)$ is, by definition, the smallest $M\leq N$ such that the $N\times N$ matrix can be factorized into a product of an $N\times M$ and an
$M\times N$ matrix. (iv) requires that, for all $P$, this algebraic rank is always equal to a corresponding analytic $\text{rank}(P(X))$, which is the trace $\Tr\otimes \tau$ of the orthogonal projection onto the image of $P(X)$.

The equivalence between (i) and (ii)/(iii)/(iv) is more or less a version of the linearization philosophy: this says that non-commutative rational functions can be realized in terms of matrices over polynomials, and hence properties for non-commutative rational functions in our operators $X_1,\dots,X_n$ should correspond to properties for matrices of polynomials in $X_1,\dots,X_n$. 

Whereas the first four items are thus equivalent versions of the statement that our operators should, on an algebraic level, be models of algebraic free indeterminates, the fourth property is of a very different type, and gives a description which property of $X_1,\dots,X_n$ ensures such a behaviour. The quantity $\Delta$ appearing there was introduced by Connes and Shlyakhtenko in their paper \cite{CS05} on the homology of von Neumann algebras. It is roughly the co-dimension of the space of 
tuples $(T_1,\dots,T_n)$ of finite rank operators on $L^2(\M,\tau)$ satisfying 
\begin{equation}\label{eq:finite-rank}
\sum_{k=1}^n[T_{k},X_{k}]=0.
\end{equation}
The requirement $\Delta(X_1,\dots,X_n)=n$ says that this space is trivial, i.e., that the only such finite rank operators are $T_1=\cdots=T_n=0$.

This should be seen as a property of the non-commutative distribution of the tuple $(X_1,\dots,X_n)$. It might not be clear how useful this characterization is, as it is not obvious how to find or exclude finite rank operators satisfying \eqref{eq:finite-rank}. In this context it is important to note that free probability theory provides a couple of tools and results which allow to address the maximality of $\Delta$ in many situations. The most straightforward situation is when the $X_i$ are selfadjoint operators. In this case,
Connes and Shlyakhtenko \cite{CS05} related $\Delta(X_1,\dots,X_n)$ with another quantity, about which we know much more. Namely, if $\delta^*(X_1,\dots,X_n)$ denotes a version the free entropy dimension of $X_1,\dots,X_n$ (which we will recall in Section \ref{sect:free-entropy-dimension}), then they showed that 
$$\delta^*(X_1,\dots,X_n)\leq \Delta(X_1,\dots,X_n)\leq n.$$
Hence the maximality of the free entropy dimension of $X_1,\dots,X_n$ implies also the maximality of $\Delta(X_1,\dots,X_n)$ and then all the other properties in Theorem \ref{thm:intro-main} also hold. Hence we get as a corollary that whenever $\delta^*(X_1,\dots,X_n)=n$ for selfadjoint operators then all rational functions as well as all full matrices over polynomials in $X_1,\dots,X_n$ are invertible as unbounded operators.
In the case where the rational function or the matrix is selfadjoint this implies then that the distribution of such non-constant functions or matrices cannot have atoms.

We have quite some tools in free probability to decide whether $\delta^*(X_1,\dots,X_n)=n$. This is, for example, the case if the $X_i$ are free and each of them has a non-atomic distribution. In particular, it is satisfied for free semicirculars. Note, however, that freeness between the $X_i$ is not necessary for having maximal free entropy dimension.

Let us also note that in relation with von Neumann algebra investigations it is still wide open whether maximality of $\delta^*(X_1,\dots,X_n)$ is equivalent to the fact that their generated von Neumann algebra is isomorphic to the free group factor $L(\FF_n)$. What we are providing with our Theorem \ref{thm:intro-main} is an algebraic version of such a statement, with $\Delta$ taking on the role of $\delta^*$. It is actually an open question whether those two quantities agree. We do in particular not know of an example where $\delta^*(X_1,\dots,X_n)<\Delta(X_1,\dots,X_n)=n$.

Since the free entropy dimension is only defined for selfadjoint operators it is of no help to address the question whether $\Delta(X_1,\dots,X_n)=n$ in more general situations. We took actually quite some efforts to prove our Theorem \ref{thm:intro-main} also without the assumption of selfadjointness - mainly with the classical result of Linnell in mind, where the operators are free Haar unitaries. Note that often one deals with non-selfadjoint situations by looking on the operators and their adjoints together; which can then be reduced to the selfadjoint case by considering the real and imaginary part of the operators. Hence if we have operators $A_1,\dots,A_n$ for which we can show that $\delta^*(A_1,\dots,A_n,A_1^*,\dots,A_n^*)=2n$, then, by our Theorem \ref{thm:intro-main}, the $A_i$ and their adjoints $A_i^*$ together give the free field in $2n$ variables, which means in particular that the $A_i$ alone give the free field in $n$ variables. Such an approach works for example for circular operators, but not for the free group case with its Haar unitary operators. In the latter case there are non-trivial algebraic relations between the operators and their adjoints, namely $U_iU_i^*=1$, and thus $U_1,\dots,U_n,U_1^*,\dots,U_n^*$ do not generate the free field in $2n$ generators. However, Linnell's result says that there are no non-trivial relations between the $U_i$ alone, not involving adjoints.
Whereas free entropy dimension arguments do not work in this case, there is another approach which, though usually also only considered for selfadjoint operators, has an adaptation to this more general situations. This relies on the existence of dual systems for our operators. Such dual systems were introduced by Voiculescu in \cite{Voi98}, and appeared also in \cite{CS05}. We adapt those ideas from the selfadjoint to more general situations and can thus show that $\Delta(U_1,\dots,U_n)=n$ for free Haar unitaries -- thus recovering Linnell's result as a special case of our general theory.

Note also that property (iii) of Theorem \ref{thm:intro-main} shows in particular that the analytic rank of $P(X)$, which apriori could be any real number between 0 and $N$, is actually an integer; since the algebraic rank is by definition an integer. This property, that $\text{rank}(P(X))$ is for all $P$ an integer, is the ``strong Atiyah property''. This was introduced by Shlyakhtenko and Skoufranis in \cite{SS15} as a generalization of such a property in the context of group algebras; the latter are related
 to the zero divisor conjecture, the Atiyah conjecture, or $l^2$-Betti numbers; for work in this context see, for example, \cite{GLSZ00,DLMSY03,PTh11}. So what our Theorem \ref{thm:intro-main} shows in particular is that operators whose division closure is the free field satisfy the strong Atiyah property. It is quite easy to see that the strong Atiyah property is actually weaker than the statements from Theorem \ref{thm:intro-main}. We will give a precise characterization, similar to Theorem \ref{thm:intro-main}, also for the strong Atiyah property. In addition to the division closure of $\C\left\langle X_{1},\dots,X_{n}\right\rangle$ this uses also its rational closure. This is, by definition, the set of all entries of inverses of matrices in $M_{N}(\C\left\langle X_{1},\dots,X_{n}\right\rangle )$.

\begin{theorem}
\label{thm:intro-main2}
Let $(\M,\tau)$ be a tracial $W^{\ast}$-probability space and $\A$ the $\ast$-algebra of affiliated unbounded operators.
For a given tuple $X=(X_{1},\dots,X_{n})$ of operators in $\M$, let
$\mathcal{R}$ be the rational closure of $\C\left\langle X_{1},\dots,X_{n}\right\rangle $. We denote the inner rank over $\mathcal{R}$ by $\rho_{\mathcal{R}}$. Then the following statements are equivalent.
\begin{enumerate}
\item $X$ has the strong Atiyah property, i.e., for any $N\in\N$ and any $P\in M_{N}(\C\left\langle x_{1},\dots,x_{n}\right\rangle )$ we have that $\rank(P(X))\in\N$.
\item The division closure of $\C\left\langle X_{1},\dots,X_{n}\right\rangle $ is a division ring.
\item The rational closure of $\C\left\langle X_{1},\dots,X_{n}\right\rangle $ is a division ring.
\item For any $N\in\N$ and any $P\in M_{N}(\C\left\langle x_{1},\dots,x_{n}\right\rangle )$ we have: if $P(X)$ is full over $\mathcal{R}$, then $P(X)\in M_{N}(\A)$ is invertible.
\item For any $N\in\N$ and any $P\in M_{N}(\C\left\langle x_{1},\dots,x_{n}\right\rangle )$ we have: $\rank(P(X))=\rho_{\mathcal{R}}(P(X))$.
\end{enumerate}
\end{theorem}

As in Theorem \ref{thm:intro-main} we have here, in (v), again an algebraic reason for the strong Atiyah property to hold. Namely, also in this more general setting the analytic rank of $P(X)$ is an integer because it has to be equal to an algebraic rank. In this case, the inner rank is not considered over the free field, but over the rational closure $\mathcal{R}$. This rational closure has to be a skew field, but it is not necessarily the free field. Let us also remark that in this case the rational closure and the division closure are the same. In contrast to Theorem \ref{thm:intro-main}, we do not have a good characterization for the situation in Theorem \ref{thm:intro-main2} in terms of the non-commutative distribution of $X_1,\dots,X_n$. It is an interesting question for further investigations to try to relate this to $\Delta$ or $\delta^*$ in cases where they are not maximal.

The paper is organized as follows.
In Section 2, we recall the basic concepts and results around the inner rank and full matrices over polynomials in noncommuting variables. Section 3 provides with Theorem \ref{thm:maximality of Delta} one of our main theorems, showing the equivalence between maximality of $\Delta$ and invertibility of full linear matrices. Section 4 provides the definition and basic facts about noncommutative rational functions and the rational closure of an algebra; in particular, the linearization idea is presented, which makes the connection between noncommutative rational functions and matrices over noncommutative polynomials. In Section 5, we recall basic facts about unbounded affiliated operators in the finite case; apply rational functions to our operators and characterize when the division or rational closure of our operators is a skew field and when it is the free skew field. Furthermore, we relate those results with regularity question for matrices in our operators. In Section 6, we provide two possibilities to check whether $\Delta$ is maximal; in the selfadjoint case this is implied by the maximality of the free entropy dimension; in general, the existence of a dual system is sufficient. The latter criterion is then used to recover in our setting the result of Linnell, that the generators of the free group generate the free field.

\section{Inner rank of matrices}\label{sec:inner rank}
In this section, we introduce the inner rank for matrices over a unital (not necessarily commutative) complex algebra $\A$; this is a generalization of the rank for matrices over $\C$ in the theory of linear algebra. We collect here some related properties and theorems that become useful in later sections. Moreover, in Section \ref{sec:rat fcts}, the analogy between the inner rank and the usual rank in linear algebra will be reinforced further by Theorem \ref{thm:free field} and Lemma \ref{lem:diagonalization}.

\subsection{Inner rank over unital algebra}
Let $\A$ be a general unital (not necessarily commutative) complex algebra.
With $M_{m,n}(\A)$ we denote the $m\times n$ matrices with entries from $\A$; and we put $M_n(\A):=M_{n,n}(\A)$.
\begin{definition}\label{def:inner-rank_full}
For any non-zero $A\in M_{m,n}(\A)$, the \emph{inner rank} of $A$ is defined as the least positive integer $r$ such that there are matrices $P\in M_{m,r}(\A)$, $Q\in M_{r,n}(\A)$ satisfying $A=PQ$. We denote this number by $\rho(A)$, and any such factorization with $r=\rho(A)$ is called a \emph{rank factorization}.

In particular, if $\rho(A)=\min\{m,n\}$, namely, if there is no such factorization with $r<\min\{m,n\}$, then $A$ is called \emph{full}. Additionally, if $A$ is a zero matrix, we define $\rho(A)=0$, hence a zero matrix is never full.
\end{definition}

It is not difficult to check that this inner rank becomes the usual rank of matrices when $\A=\C$. So the inner rank generalize consistently the notion of rank from $M_n(\C)$ to matrices over a general algebra. In the following, we list some properties for the inner rank, which depict more analogies between the inner rank and the usual rank in linear algebra. Before we begin the promised list, we introduce first stably finite algebras.

\begin{definition}\label{def:stably finite}
$\A$ is called \emph{stably finite} (or \emph{weakly finite}) if for any $n\in\mathbb{N}$, and all $A,B\in M_{n}(\A)$ the equation $AB=\1_{n}$ implies that also $BA=\1_{n}$ holds.
\end{definition}

\begin{proposition}
\label{prop:invertible minor}(See \cite[Proposition 5.4.6]{Coh06}) Suppose that $\A$ is stably finite. Let $A\in M_{m+n}(\A)$ be of the form
\[A=\begin{pmatrix}B & C\\D & E\end{pmatrix},\]
where $B\in M_{m}(\A)$, $C\in M_{m,n}(\A)$, $D\in M_{n,m}(\A)$ and $E\in M_{n}(\A)$. If $B$ is invertible, then $\rho(A)\geqslant m$, with equality if and only if $E=DB^{-1}C$.
\end{proposition}

Note that if $\A$ is particularly a division ring, then $\A$ is stably finite. Moreover, in such a case, if additionally $n=1$, then Proposition \ref{prop:invertible minor} becomes the well-known Schur's lemma. In the following lemma, we state Schur's lemma in full generality, which will be used for proving Theorem \ref{thm:Atiyah-2}.

\begin{lemma}
\label{lem:Schur complement}Suppose that $\A$ is a unital algebra. Let $A\in M_{m+n}(\A)$ be of the form
\[A=\begin{pmatrix}B & C\\D & E\end{pmatrix},\]
where $B\in M_{m}(\A)$, $C\in M_{m,n}(\A)$, $D\in M_{n,m}(\A)$ and $E\in M_{n}(\A)$. If $B$ is invertible, then $A$ is invertible in $M_{m+n}(\A)$ if and only if the Schur complement $E-DB^{-1}C$ is invertible in $M_{n}(\A)$.
\end{lemma}

Back to the inner rank, we have the important following theorem. It is an analogue of the theorem that claims that a matrix of rank $r$ in $M_{n}(\C)$ always has a $r\times r$ non-singular submatrix.

\begin{theorem}
\label{thm:full minor}(See \cite[Theorem 5.4.9]{Coh06}) Suppose that the set of all square full matrices over $\A$ is closed under products and diagonal sums. Then for any $A\in M_{m,n}(\A)$, there exists a square submatrix of $A$ which is a full matrix over $\A$ of dimension $\rho(A)$. Moreover, $\rho(A)$ is the maximal dimension for such submatrices.
\end{theorem}

In the remaining part of this section, we set $\A=\C\left\langle x_{1},\dots,x_{d}\right\rangle$, the algebra of noncommutative polynomials in (formal) non-commuting variables $x_{1},\dots,x_{d}$. Then the requirements in Theorem \ref{thm:full minor} can be verified (since $\C\left\langle x_{1},\dots,x_{d}\right\rangle$ is a Sylvester domain, see \cite[Section 5.5]{Coh06} for details).
Moreover, as a corollary we have the following proposition, which is Lemma 4 in Section 4 in \cite{CR94}.

\begin{proposition}
\label{prop:full minor-one column omitted}Let $A\in M_{n}(\C\left\langle x_{1},\dots,x_{d}\right\rangle)$ be given in the form $A=(B \,\, b)$, where $b$ is the last column of $A$ and $B$ is the remaining block. Assume that $A$ is full, then we can choose $n-1$ rows of $B$ to form a full $(n-1)\times(n-1)$ matrix.
\end{proposition}

To see this, note that $B\in M_{n,n-1}(\C\left\langle x_{1},\dots,x_{d}\right\rangle)$ is also a full matrix, which is not difficult to see from the definition; then applying Theorem \ref{thm:full minor} to $B$ yields the above proposition.

\subsection{Linear matrices}

Now, consider a matrix of form
\[A=\begin{pmatrix}P & \0\\Q & R\end{pmatrix}\in M_{n}(\C\left\langle x_{1},\dots,x_{d}\right\rangle),\]
which has a zero block of size $r\times s$ and blocks $P$, $Q$, $R$ of sizes $r\times (n-s)$, $(n-r)\times (n-s)$, $(n-r)\times s$, respectively. Then we have the factorization
\[A=\begin{pmatrix}P & \0\\Q & R\end{pmatrix}=\begin{pmatrix}P & \0\\\0 & \1_{n-r}\end{pmatrix}\begin{pmatrix}\1_{n-s} & \0\\Q & R\end{pmatrix}.\]
So $A$ has been expressed as a product of an $n\times(2n-r-s)$ matrix and an $(2n-r-s)\times n$ matrix; this allows us to conclude that $\rho(A)\leq 2n-r-s$. Therefore, if the size of the zero block of $A$ satisfies $r+s>n$, then we have $\rho(A)<n$, which means that $A$ is not full. Such matrices are called hollow matrices.

\begin{definition}
\label{def:hollow}A matrix in $M_{n}(\C\left\langle x_{1},\dots,x_{d}\right\rangle )$ is called \emph{hollow} if it has an $r\times s$ block of zeros with $r+s>n$.
\end{definition}

In general, a non-full $A\in M_{n}(\C\left\langle x_{1},\dots,x_{d}\right\rangle)$ may not have any zero blocks or submatrices. However, we will be mostly interested in special matrices for which we can say more.

\begin{definition}
A matrix $A\in M_{n}(\C\left\langle x_{1},\dots,x_{d}\right\rangle )$ is called \emph{linear} if it can be written in the form $A=A_{0}+A_{1}x_{1}+\cdots+A_{d}x_{d}$, where $A_{0},A_{1}\dots,A_{d}$ are $n\times n$ matrices over $\C$. Note that we allow also a constant term in a general linear matrix. We call the non-constant part $A-A_{0}=A_{1}x_{1}+\cdots+A_{d}x_{d}$ the \emph{homogeneous part of $A$}.
\end{definition}

For linear matrices we have the following theorem to reveal their zero block structure. That is, for a linear matrix $A$, it is always possible to bring $A$ into a form with some possible zero block. From this block structure, we can see that $A$ is not full if and only if $A$ is hollow, up to multiplying by scalar-valued invertible matrices.

\begin{theorem}
\label{thm:zero block}
Let $A$ be a linear matrix in $M_{n}(\C\left\langle x_{1},\dots,x_{d}\right\rangle)$. Then there exist invertible matrices $U,V\in M_{n}\left(\C\right)$ and an integer $s\in[0,\rho(A)]$ such that
\begin{equation}\label{eq:zero block}
UAV=\begin{pmatrix}B & \0\\
C_{1} & C_{2}
\end{pmatrix}
\end{equation}
where $B\in M_{n-s,\rho(A)-s}(\C\left\langle x_{1},\dots,x_{d}\right\rangle)$ has inner rank $\rho(B)=\rho(A)-s$.
\end{theorem}

Actually, by the above theorem, we see that the zero block has size $(n-s)\times(n-\rho(A)+s)$. Therefore, if $A$ is not full, i.e., $\rho(A)<n$, then the right hand side of \eqref{eq:zero block} is hollow since $2n-\rho(A)>n$. For a proof of the above theorem, we refer to \cite[Corollary 6.3.6]{Coh95} or \cite[Theorem 5.8.8]{Coh06}. The statements therein do not address the inner rank of the block $B$ of $UAV$, but from the corresponding proofs it is easy to extract this additional statement.

We finish this subsection by mentioning another interesting criterion for the fullness of linear matrices that was given in \cite{GGOW16}.






\begin{proposition}\label{prop:rank-decreasing}
Consider a linear matrix $A = A_{1}x_{1}+\cdots+A_{d}x_{d}$ in $M_{n}(\C\left\langle x_{1},\dots,x_{d}\right\rangle )$ with zero constant part. Then $A$ is full if and only if the associated \emph{quantum operator}
$$\cL:\ M_n(\C) \to M_n(\C), \qquad B \mapsto \sum^d_{i=1} A_i B A_i^\ast$$
is \emph{nowhere rank-decreasing}, i.e., there is no positive semidefinite matrix $B \in M_n(\C)$ for which $\rank(\cL(B)) < \rank(B)$ holds.
\end{proposition}

This connects fullness very nicely with concepts that are used, for instance, in \cite{AjEK18,AEK18}; we will say more about this in Section \ref{subsec:regularity_linear_matrices}.

\subsection{Central eigenvalues}

Recall that for a matrix $A\in M_n(\C)$, its spectrum $\sigma(A)$ is defined as
\[
\big\{\lambda \in \C \bigm| \text{$A - \lambda \1_n$ is not invertible in }M_n(\C)\big\}.
\]
That is,
\[
\sigma(A)= \big\{\lambda \in \C \bigm| \rho_\C(A - \lambda \1_n)<n \big\},
\]
where $\rho_\C$ denotes the (inner) rank of $A$ over $\C$. This is simply because the invertibility of $A$ is equivalent to its rank being maximal.

Now, we consider a unital complex algebra $\A$. In Lemma \ref{lem:diagonalization} we will see that the fullness, i.e., the maximality of inner rank over $\A$ is indeed equivalent to the invertibility if $\A$ is additionally a division ring. So if we consider fullness as invertibility over $\A$, it is natural to define a spectrum as the following.

\begin{definition}\label{def:central eigenvalue}
Let $\A$ be a unital complex algebra. For each square matrix $A$ over $\A$, say $A\in M_n(\A)$ for some $n\in\N$, we define
\[
\sigma^\full_\A(A) := \big\{\lambda \in \C \bigm| \rho_\A(A - \lambda \1_n)<n \big\},
\]
where $\rho_\A$ denotes the inner rank over $\A$ and $\1_n$ stands for the unit element in $M_n(\A)$. The numbers $\lambda\in\sigma^\full_\A(A)$ are called \emph{central eigenvalues} of $A$.
\end{definition}

This concept of central eigenvalues was introduced in \cite[Section 8.4]{Coh85}; we thank Konrad Schrempf for bringing this reference to our attention.

We specialize our considerations now to the relevant case $\A = \C\langle x_1,\dots,x_d\rangle$ and will in the following write $\sigma^\full(A)$ for $\sigma^\full_{\C\langle x_1,\dots,x_d\rangle}(A)$ for a square matrix $A$ over $\C\langle x_1,\dots,x_d\rangle$. 

\begin{proposition}\label{prop:finite spectrum}
(See \cite[Proposition 8.4.1]{Coh85}) Let $A$ be a matrix in $M_n(\C\langle x_1,\dots,x_d \rangle)$ for some $n\in\N$. Then $A$ has at most $n$ central eigenvalues.
\end{proposition}

Cohn's proof of this theorem relied on some involved algebraic considerations.
In Section \ref{subsec:regularity_polynomials_matrices}, we will give an alternative proof of Proposition \ref{prop:finite spectrum}, based on our results.

For linear matrices over $\C\langle x_1,\dots,x_n \rangle$, we can say more about their central eigenvalues.

\begin{proposition}\label{prop:spectrum of linear matrices}
Let any linear $A \in M_n(\C\langle x_1,\dots,x_d \rangle)$ of the form $A = A_0 + A_1 x_1 + \cdots + A_d x_d$ with $A_0,A_1,\dots,A_d\in M_n(\C)$ be given. Then the following statements hold.
\begin{enumerate}
 \item\label{it:spectrum of linear matrices_1} We have that $\sigma^\full(A) \subset \sigma(A_0)$, where $\sigma(A_0)$ is the usual spectrum of $A_0$ consisting of all eigenvalues of $A_0$.
 \item\label{it:spectrum of linear matrices_2} If the homogeneous part $A - A_0$ of $A$ is full, then $\sigma^\full(A) = \emptyset$.
\end{enumerate}
\end{proposition}

\begin{proof}
Let $\lambda\in\sigma^{\full}(A)$ be given. By definition, this means that $A-\lambda \1_n$ is not full, so that Theorem \ref{thm:zero block} guarantees the existence of invertible matrices $U,V\in M_N(\C)$ for which
$$U (A - \lambda \1_n) V = U (A_0 - \lambda \1_n) V + \sum_{j=1}^d (U A_j V) x_j$$
is hollow. Due to linearity, this enforces both $U (A_0 - \lambda \1_n) V$ and $\sum_{j=1}^d (U A_j V) x_j$ to be hollow.
Now, on the one hand, it follows that neither $U (A_0 - \lambda \1_N) V$ nor $A_0 - \lambda \1_N$, thanks to the invertibility of $U$ and $V$, can be invertible; thus, we infer that $\lambda \in \sigma(A_0)$, which shows the validity of \ref{it:spectrum of linear matrices_1}.
On the other hand, we see that neither $\sum_{j=1}^d (U A_j V) x_j$ nor $\sum_{j=1}^d A_j x_j$, by the invertibility of $U$ and $V$, can be full; thus, if the homogeneous part of $A$ is assumed to be full, that contradiction rules out the existence of $\lambda\in\sigma^{\full}(A)$, which proves \ref{it:spectrum of linear matrices_2}. 
\end{proof}

\section{Maximality of $\Delta$ and triviality of kernels of linear full matrices}
\label{sec:maximality of Delta}

Let $(\M,\tau)$ be a tracial $W^\ast$-probability space (i.e., a von Neumann algebra $\M$ that is endowed with a faithful normal tracial state $\tau: \M \to \C$) and consider a tuple $X=(X_1,\dots,X_n)$
of (not necessarily selfadjoint) noncommutative random variables in $\M$.
A quantity $\Delta(X)$ was introduced in \cite{CS05} as
$$\Delta(X) := n - \dim_{M \ootimes M^{\operatorname{op}}} \overline{\Big\{(T_1,\dots,T_n) \in \F(L^2(\M,\tau))^n \mid \sum^n_{j=1} [T_j, J X_j J] = 0\Big\}}^{\operatorname{HS}}.$$
Here, we denote by $\F(L^2(\M,\tau))$ the ideal of all finite rank operators on $L^2(\M,\tau)$ and by $J$ Tomita's conjugation operator, i.e., the conjugate-linear map $J: L^2(\M,\tau) \to L^2(\M,\tau)$ that extends isometrically the conjugation $x \mapsto x^\ast$ on $\M$; the closure is taken with respect to the Hilbert-Schmidt norm.
Note that for the Hochschild homology of $\C\langle X_1,\dots,X_n \rangle$ with coefficients in the Hilbert-Schmidt operators on $L^2(\M,\tau)$, which was studied in \cite{CS05}, tuples $(T_1,\dots,T_n)$ of finite rank operators on $L^2(\M,\tau)$ satisfying $\sum^n_{j=1} [T_j, J X_j J] = 0$ are precisely the Hochschild $1$-cycles.

Our main goal in this section is to prove the equivalence of items (ii) and (v) in Theorem \ref{thm:intro-main}, i.e., to see that the maximality of $\Delta$ is the property of a tuple $X=(X_1,\dots,X_n)$ which decides upon whether evaluations of full linear matrices in $X$ have no kernel (which, as we will see in Section \ref{subsec:free field realization}, is the same as saying that those are invertible as unbounded operators).

\begin{theorem}\label{thm:maximality of Delta}
Let $(\mathcal{M},\tau)$ be a tracial $W^{\ast}$-probability space and $X=(X_{1},...,X_{n})$ a tuple of (not necessarily selfadjoint) random variables in $\mathcal{M}$. Then the following are equivalent.
\begin{enumerate}
\item $\Delta(X)=n$.
\item For any $N\in\mathbb{N}$ and any linear matrix $A\in M_{N}(\C\left\langle x_{1},\dots,x_{n}\right\rangle)$ we have: if there are $f\in\ker A(X)$ and $e\in\ker A(X)^{\ast}$ such that both $e$ and $f$ have $\mathbb{C}-$linear independent components, then $A=0$.
\item For any $N\in\mathbb{N}$ and any linear full matrix $A\in M_{N}(\C\left\langle x_{1},\dots,x_{n}\right\rangle)$ we have: $\ker A(X)=\{0\}$.
\end{enumerate}
\end{theorem}

The equivalence of (i) and (iii) is the first step for proving Theorem \ref{thm:intro-main}. In Section \ref{sec:affiliated_operators_and_Atiyah}, we will extend this list of equivalent properties to the full list of Theorem \ref{thm:intro-main}.

\begin{remark}
Before giving the proof, the following remarks are in order.
\begin{enumerate}
\item Suppose that $\Delta(X_1,\dots,X_n)<n$. By definition, this means that there is a tuple $(T_1,\dots,T_n) \neq (0,\dots,0)$ of finite rank operators in $B(L^2(\M,\tau))$ with the property that $\sum^n_{j=1} [T_j, J X_j J] = 0$; we infer from the latter that $(JT_1J,\dots,JT_nJ)$, which is again a non-trivial tuple of finite rank operators on $L^2(\M,\tau)$, satisfies $\sum^n_{j=1} [J T_j J, X_j] = 0$.
\item There is a natural anti-linear involution $\ast$ on $\C\langle x_1,\dots,x_n\rangle$ determined by $1^\ast=1$ and $x_j^\ast = x_j$ for $j=1,\dots,n$. Then for any matrix $A\in M_{N}(\C\left\langle x_{1},\dots,x_{n}\right\rangle)$ and any tuple $X$ over $M$, we see that
\[
A(X)^\ast=A^\ast(X^\ast),
\]
where $X^\ast:=(X_1^\ast,\cdots,X_n^\ast)$.
\end{enumerate}
\end{remark}

\subsection{Proof of Theorem \ref{thm:maximality of Delta}}

First, we need to recall the following well-known result. The interested reader can find a detailed proof of this statement in \cite{MSW17}.

\begin{lemma}\label{lem:kernels}
Let $X$ be an element of any tracial $W^\ast$-probability space $(\M,\tau)$ over some complex Hilbert space $H$. Let $p_{\ker(X)}$ and $p_{\ker(X^\ast)}$ denote the orthogonal projections onto $\ker(X)$ and $\ker(X^\ast)$, respectively.

The projections $p_{\ker(X)}$ and $p_{\ker(X^\ast)}$ belong both to $\M$ and satisfy
$$\tau(p_{\ker(X)}) = \tau(p_{\ker(X^\ast)}).$$
Thus, in particular, if $\ker(X)$ is non-zero, then also $\ker(X^\ast)$ is a non-zero subspace of $H$.
\end{lemma}

Before giving the proof of Theorem \ref{thm:maximality of Delta}, we single out the following two lemmas. These lemmas highlight an explicit way on how we can construct finite rank operators satisfying the commutator relation from linear matrices with kernel vectors and vice versa.

\begin{lemma}
\label{lem:finite rank operators}Given a linear matrix $A=A^{(0)}+A^{(1)} x_1+\cdots+A^{(n)}x_n$ in $M_{N}(\C\left\langle x_{1},\dots,x_{n}\right\rangle)$
with vectors $f=(f_1,\dots,f_N)\in\ker A(X)$ and $e=(e_1,\dots,e_N)\in\ker A(X)^{\ast}$, we define
finite rank operators
\begin{equation}\label{eq:finite rank operators}
T_{k}:=\sum_{i,j=1}^{N}A_{ij}^{(k)}\left<\cdot,e_{i}\right>f_{j},\ k=0,\dots,n,
\end{equation}
where $A_{ij}^{(k)}$ ($i,j=1,\dots,N$) are the entries of $A^{(k)}$ for $k=0,\dots,n$ and where, for $e,f\in L^2(\mathcal{M},\tau)$, $\left<\cdot,e\right>f$ denotes the operator of rank 1, which maps any $v\in L^{2}(\mathcal{M},\tau)$ to $\left<v,e\right>f$.
Then $T_{1},\dots,T_{n}$ satisfy
\[
\sum\limits _{k=1}^{n}[T_{k},X_{k}]=0.
\]
\end{lemma}

\begin{proof}
First, we write $A(X)f=0$ in entries, namely,
\[
\sum_{j=1}^{N}A_{ij}^{(0)}f_{j}+\sum_{k=1}^{n}\sum_{j=1}^{N}A_{ij}^{(k)}X_{k}f_{j}=0,\ \forall i=1,\dots,N.
\]
Then for each vector $e_{i}$ ($i=1,\dots,N$)  we get an equality
of finite rank operators on $L^{2}(\mathcal{M},\tau)$ from the above, that is,
\[
\sum_{j=1}^{N}A_{ij}^{(0)}\left<\cdot,e_{i}\right>f_{j}+\sum_{k=1}^{n}\sum_{j=1}^{N}A_{ij}^{(k)}\left<\cdot,e_{i}\right>X_{k}f_{j}=0,\ \forall i=1,\dots,N.
\]
Summing those equations over the index $i$ gives
\[
\sum_{i,j=1}^{N}A_{ij}^{(0)}\left<\cdot,e_{i}\right>f_{j}+\sum_{k=1}^{n}\sum_{i,j=1}^{N}A_{ij}^{(k)}\left<\cdot,e_{i}\right>X_{k}f_{j}=0,
\]
or equivalently,
\begin{equation}
T_{0}+\sum_{k=1}^{n}X_{k}T_{k}=0.\label{eq:commutator-1}
\end{equation}
Similarly, from $A(X)^{\ast}e=0$, i.e.,
\[
\sum_{j=1}^{N}\overline{A_{ji}^{(0)}}e_{j}+\sum_{k=1}^{n}\sum_{j=1}^{N}\overline{A_{ji}^{(k)}}X_{k}^{\ast}e_{j}=0,\ \forall i=1,\dots,N,
\]
we have
\[
\sum_{j=1}^{N}\left<\cdot,\overline{A_{ji}^{(0)}}e_{j}\right>f_{i}+\sum_{k=1}^{n}\sum_{j=1}^{N}\left<\cdot,\overline{A_{ji}^{(k)}}X_{k}^{\ast}e_{j}\right>f_{i}=0,\ \forall i=1,\dots,N,
\]
which yields
\[
\sum_{i,j=1}^{N}A_{ji}^{(0)}\left<\cdot,e_{j}\right>f_{i}+\sum_{k=1}^{n}\sum_{i,j=1}^{N}A_{ji}^{(k)}\left<X_{k}\cdot,e_{j}\right>f_{i}=0.
\]
Actually, that is
\begin{equation}
T_{0}+\sum_{k=1}^{n}T_{k}X_{k}=0.\label{eq:commutator-2}
\end{equation}
Therefore, combining (\ref{eq:commutator-1}) and (\ref{eq:commutator-2}),
we conclude that $\sum_{k=1}^{n}[T_{k},X_{k}]=0$.
\end{proof}

\begin{lemma}\label{lem:linear matrix}
Suppose that $X=(X_{1},\dots,X_{n})$ is a tuple of operators in $\mathcal{M}$ and $(T_1,\dots,T_n)$ is a tuple of finite rank operators on $L^2(\M,\tau)$ satisfying
\begin{equation}\label{eq:commutator-3}
\sum_{k=1}^n[T_{k},X_{k}]=0.
\end{equation}
Let $f=(f_1,\dots,f_N)$ be an orthonormal family which spans the space of the sum of subspaces $\im T_k+\im T^\ast_k$ over all $k=1,\dots,n$. Let us write each $T_k$ as
\begin{equation}\label{eq:finite rank operators-2}
T_k = \sum_{i,j=1}^N A_{ij}^{(k)}\left<\cdot,f_i\right>f_j,\ k=1,\dots,n.
\end{equation}
Consider the linear matrix
\begin{equation}\label{eq:linear matrix}
A:=A^{(0)}-(A^{(1)}x_1+\cdots+A^{(n)}x_n),
\end{equation}
where $A^{(k)}:=(A^{(k)}_{ij})_{i,j=1}^N\in M_N(\C)$, for $k=1,\dots,n$, and $A^{(0)}$ is given by
$$
A^{(0)}:=\sum_{k=1}^n B^{(k)}A^{(k)},\qquad\text{where}\qquad
B^{(k)}:=(\left<X_kf_i,f_j\right>)_{i,j=1}^N\quad (k=1\dots,n).
$$
Then $A(X)$
satisfies
$A(X)f=0$ and $A(X)^\ast f=0$.
\end{lemma}

\begin{proof}
Substituting the form of $T_k$ from \eqref{eq:finite rank operators-2} into the relation \eqref{eq:commutator-3}, we have
\[
\sum_{k=1}^n\sum_{i,j=1}^N A_{ij}^{(k)}\left<X_k\cdot,f_i\right>f_j=\sum_{k=1}^n\sum_{i,j=1}^N A_{ij}^{(k)}\left<\cdot,f_i\right>X_k f_j.
\]
Applying this to the vectors $f_p$, for $p=1,\dots,N$, we obtain
$$\sum_{k=1}^n\sum_{i,j=1}^N B_{pi}^{(k)}A_{ij}^{(k)}f_j=
{\sum_{k=1}^n\sum_{i,j=1}^N A_{ij}^{(k)}\left<X_k f_p,f_i\right>f_j=\sum_{k=1}^n\sum_{i,j=1}^N A_{ij}^{(k)}\left<f_p,f_i\right>X_k f_j} =\sum_{k=1}^n\sum_{j}^N A_{pj}^{(k)}X_k f_j,$$
or, by using $A^{(0)}$,
\begin{equation}\label{eq:commutator-4}
 \sum_{j=1}^N A_{pj}^{(0)}f_j=\sum_{k=1}^n\sum_{j}^N A_{pj}^{(k)}X_k f_j .
\end{equation}
This is true for each
for $p=1,\dots,N$; so we have $A(X)f=0$, as desired.

Moreover, by taking inner products of both sides of \eqref{eq:commutator-4} with vectors $f_q$, $q=1,\dots,N$, we obtain
$$
{\sum_{j=1}^N A_{pj}^{(0)}\left<f_j,f_q\right>=\sum_{k=1}^n\sum_{j}^N A_{pj}^{(k)}\left<X_k f_j,f_q\right>}
$$
or
\begin{equation}\label{eq:constant term} 
A_{pq}^{(0)}=\sum_{k=1}^n\sum_{j}^N A_{pj}^{(k)}B_{jq}^{(k)} 
\end{equation}
for $p,q=1,\dots,N$. That actually says that 
$$\sum_{k=1}^n B^{(k)}A^{(k)}=A^{(0)}=\sum_{k=1}^n A^{(k)}B^{(k)}.$$

Finally, we want to verify the remaining part, i.e., $A(X)^\ast f=0$. For that purpose, we consider
\[
\sum_{k=1}^n X_k^\ast T_k^\ast=\sum_{k=1}^n T_k^\ast X_k^\ast,
\]
which comes from taking conjugation of \eqref{eq:commutator-3}. We replace $T_k^\ast$ in the above equality by
\[
T_k^\ast=\sum_{i,j=1}^N \overline{A_{ij}^{(k)}}\left<\cdot,f_j\right>f_i,
\]
for $k=1,\dots,n$; then we have
\[
\sum_{k=1}^n\sum_{i,j=1}^N\overline{A_{ij}^{(k)}}\left<\cdot,f_j\right>X_k^\ast f_i=\sum_{k=1}^n\sum_{i,j=1}^N\overline{A_{ij}^{(k)}}\left<X_k^\ast\cdot,f_j\right>f_i.
\]
Applying this to $f_l$, for $l=1,\dots,N$, gives
$$
 \sum_{k=1}^n\sum_{i=1}^N\overline{A_{il}^{(k)}}X_k^\ast f_i={\sum_{k=1}^n\sum_{i,j=1}^N\overline{A_{ij}^{(k)}}\left<f_l,f_j\right>X_k^\ast f_i=\sum_{k=1}^n\sum_{i,j=1}^N\overline{A_{ij}^{(k)}}\left<X_k^\ast f_l,f_j\right>f_i}=
\sum_{k=1}^n\sum_{i,j=1}^N\overline{A_{ij}^{(k)}}\overline{B_{jl}^{(k)}}f_i
$$
or
$$ \sum_{k=1}^n\sum_{i=1}^N\overline{A_{il}^{(k)}}X_k^\ast f_i = \sum_{i=1}^N\overline{A_{il}^{(0)}}f_i,
$$
where we have used \eqref{eq:constant term} in the last step. The last equality is exactly $A(X)^\ast f=0$, as desired.
\end{proof}

Now we can give the proof of Theorem \ref{thm:maximality of Delta}.

\begin{proof}[Proof of Theorem \ref{thm:maximality of Delta}]
We prove our theorem by showing (i)$\Longleftrightarrow$(ii) and (ii)$\Longleftrightarrow$(iii).

First, we want to prove (i)$\Longrightarrow$(ii). Let $A$ be a linear matrix in $M_{N}(\mathbb{C}\left\langle x_{1},\dots,x_{n}\right\rangle)$ with vectors $f\in\ker A(X)$, $e\in\ker A(X)^{\ast}$ such that both $e$ and $f$ have $\mathbb{C}-$linear independent components, as given in (ii). Then we need to show that $A=0$.
We use Lemma \ref{lem:finite rank operators} to define finite rank operators $T_k$ ($k=1,\dots,n$) as in \eqref{eq:finite rank operators}. Those operators satisfy then the relation $\sum_{k=1}^{n}[T_{k},X_{k}]=0$. Hence, from our hypothesis (i), i.e., $\Delta(X)=n$, we conclude that $T_{k}=0$ for all $k=1,
\dots,n$. In particular, we have
\[
T_k(e_{p})=\sum_{i,j=1}^{N}A_{ij}^{(k)}\left<e_{p},e_{i}\right>f_{j}=0,\quad \text{for all $p=1,\dots,N$,}
\]
and furthermore,
\[
\left<T_k(e_{p}),f_{q}\right>=\sum_{i,j=1}^{N}\left<e_{p},e_{i}\right>A_{ij}^{(k)}\left<f_{j},f_{q}\right>=0,\quad\text{for all $p,q=1,\dots,N$},
\]
for $k=1,\dots,n$. Denoting by $E:=(\left<e_{p},e_{i}\right>)_{p,i=1}^{N}$ and $F:=(\left<f_{j},f_{q}\right>)_{j,q=1}^{N}$ the Gram matrices of $e$ and $f$, we write the above as
\[
EA^{(k)}F=0,
\]
for $k=1,\dots,n$. Since both $e$ and $f$ consist of $\mathbb{C}-$linear
independent vectors, $E$ and $F$ are invertible. Therefore, $A^{(k)}=0$ for all
$k=1,\dots,n$. This reduces then $A(X)f=0$ to $A^{(0)}f=0$.
However, as $f$ is a $\mathbb{C}-$linear independent family, this implies that $A^{(0)}=0$.
This completes the proof for the part (i)$\Longrightarrow$(ii).

Next, we want to show (ii)$\Longrightarrow$(i). In order to prove $\Delta(X)=n$, let $(T_1,\dots,T_n)$ be a tuple of finite rank operators on $L^2(\M,\tau)$ satisfying $\sum_{k=1}^{n}[T_k,X_k]=0$. Then by Lemma \ref{lem:linear matrix}, the linear matrix $A$ defined as in \eqref{eq:linear matrix} satisfies $A(X)f=0$ and $A(X)^\ast f=0$, where $f$ is an orthonormal family spanning the space of the sum of $\im T_k + \im T_k^ \ast$ over all $k=1,\dots,n$. Hence $A=0$ by our hypothesis (ii), which enforces, immediately from \eqref{eq:finite rank operators-2}, that $T_k=0$ for all $k=1\dots,n$ . Hence, $\Delta(X)=n$.

Now, we want to prove (ii)$\Longrightarrow$(iii). We proceed by induction on the matrix size $N$ which is considered in (iii) of Theorem \ref{thm:maximality of Delta}.
First, we prove by contradiction the result when the dimension $N=1$. Suppose $f\in L^{2}(\mathcal{M},\tau)$
is a nonzero vector such $A(X)f=0$, then there also exists a nonzero
vector $e\in L^{2}(\mathcal{M},\tau)$ such that $A(X)^\ast e=0$ by Lemma \ref{lem:kernels}.
Therefore, $A=0\in\C$ by our hypothesis (ii), which is a contradiction with $A$ is full, i.e., $A\neq0$.

In the following, we assume that the result holds for dimension $N-1$
and we want to prove it for dimension $N$. Let $A$ be a linear full
matrix over $\mathbb{C}\left\langle x_{1},\dots,x_{n}\right\rangle $
such that $A(X)f=0$ for some vector $f\neq0$ in $L^{2}(\mathcal{M},\tau)^{N}$.
Then the components of $f$ have to be $\mathbb{C}-$linear independent by the following reason.
If these components are not $\mathbb{C}-$linear independent, then we can find an invertible matrix $U\in M_{N}(\mathbb{C})$ such that
\[
Uf=\begin{pmatrix}f'\\
0
\end{pmatrix},
\]
where $f'\neq0$ in $L^{2}(\mathcal{M},\tau)^{N-1}$ as $f\neq0$.
Putting $AU^{-1}=\begin{pmatrix}B & b\end{pmatrix}$ in the corresponding
block structure, we have $B(X)f'=0$. We can further choose $N-1$
rows of $B$ to form a full matrix $B'\in M_{N-1}(\mathbb{C}\left\langle x_{1},\dots,x_{n}\right\rangle )$ by Proposition \ref{prop:full minor-one column omitted} since $AU^{-1}$ is a full matrix. 
Hence, we have $B'(X)f'=0$, which yields $f'=0$ by the induction hypothesis. This is a contradiction
with $f'\neq0$ and thus the components of $f$ are $\mathbb{C}-$linear independent.

For $A(X)^\ast$, by Lemma \ref{lem:kernels}, there also exists a vector $e\neq0$ in $L^{2}(\mathcal{M},\tau)^{N}$
such that
$$A(X)^\ast e=A^\ast(X^\ast)e=0.$$
If the components of $e$ are not $\mathbb{C}-$linear independent, then
by a similar argument as the case $A(X)f=0$, we can construct
a linear full matrix $B'\in M_{N-1}(\mathbb{C}\left\langle x_{1},\dots,x_{n}\right\rangle)$ out of
$A^\ast$ (note $A^\ast$ is full as $A$ is full) such that $B'(X^\ast)$ also has a nontrivial kernel. Hence $B'(X^\ast)^\ast=(B')^\ast(X)$ has a nontrivial kernel. However, since $(B')^\ast$ is a linear full matrix of dimension $N-1$, the kernel of
$(B')^\ast(X)$ is trivial by the induction hypothesis. This yields a contradiction and thus we see that $e$ also has $\mathbb{C}-$linear independent components.
Therefore, $A=0$ follows from hypothesis (ii), which is a contradiction with the fullness of $A$.

Finally, we want to show (iii)$\Longrightarrow$(ii). Suppose that $N\geq 1$, $A\in M_{N}(\C\left\langle x_{1}\dots,x_{n}\right\rangle)$ is a linear matrix such that $f\in\ker A(X)$ and $e\in\ker A(X)^{\ast}$ both have $\mathbb{C}-$linear independent components. Our goal is to prove $A=0$, so we assume $\rho(A)>0$ in order to achieve some contradiction. First, we note that $A$ is not full, otherwise $\ker A(X)=\{0\}$ by hypothesis (iii). This is a contradiction since $f\in\ker A(X)$ has $\mathbb{C}-$linear independent components. Therefore, we may additionally assume $\rho(A)<N$. By Theorem \ref{thm:zero block}, there are invertible matrices $U$ and $V$ in $M_{N}(\C)$ such that
\[
UAV=\begin{pmatrix}B & \0\\C_{1} & C_{2}\end{pmatrix}
\]
where $B\in M_{N-S,\rho(A)-S}(\C\left\langle x_{1},\dots,x_{n}\right\rangle )$ has inner rank $\rho(B)=\rho(A)-S$.

If $S<\rho(A)$, i.e., the block $B$ does not disappear in the above form of $UAV$, then let us write
\[
V^{-1}f=\begin{pmatrix}f'\\f''\end{pmatrix},
\]
where $f'\in L^2(\M,\tau)^{\rho(A)-S}$. Clearly we have $B(X)f'=0$ and we consider this equality rather than $A(X)f=0$. By consulting Theorem \ref{thm:full minor}, we can choose $\rho(A)-S$ rows of $B$ to form a linear full matrix $B'\in M_{\rho(A)-S}(\C\left\langle x_{1},\dots,x_{n}\right\rangle)$. It then follows that $B'(X)f'=0$ and thus $f'=0$ by our hypothesis (iii). However, this is impossible since $f$ has $\C$-linear independent components and $V\in M_{N}(\C)$ is invertible.

So it remains to deal with the case that $S=\rho(A)$. In such a case, Theorem \ref{thm:zero block} actually says that
\[
UAV=\begin{pmatrix}\0\\C_{2}\end{pmatrix},
\]
where $C_2\in M_{\rho(A),N}(\C\left\langle x_{1},\dots,x_{n}\right\rangle)$. From $A(X)^\ast e=0$, we see that
\[
(UAV)^\ast(X^\ast)(U^\ast)^{-1}e=V^\ast A^\ast(X^\ast)U^\ast(U^\ast)^{-1}e=V^\ast A(X)^\ast e=0.
\]
That is,
\[
\begin{pmatrix}\0 & C_{2}^\ast(X^\ast)\end{pmatrix}(U^\ast)^{-1}e=0.
\]
Similarly as before, we can build out of $C_2^\ast$ a linear full matrix which has nontrivial kernel. This gives a contradiction and our proof is completed.

\end{proof}

\section{Noncommutative Rational functions and rational closure}\label{sec:rat fcts}

In this section, we will give an introduction to noncommutative rational functions. One crucial fact is that to each rational function we can associate a representation using linear full matrices; this is also known as ``linearization trick''. Another fact on rational functions is that they can be used to diagonalize a matrix over polynomials. These two facts on rational functions then allow, in Section \ref{subsec:free field realization}, to give more equivalent descriptions for tuples of operators with maximal $\Delta$, and thus get the full version of Theorem \ref{thm:intro-main}.
Furthermore, in the last subsection, the rational closure and the division closure are introduced; they will take over the role of rational functions when we consider the Atiyah property in Section \ref{subsec:Atiyah}.

\subsection{Noncommutative rational functions and their linear representations}

It is well-known that for a given integral domain $\A$, one can construct the smallest field containing $\A$, which is called the field of fractions. In particular, if $\A=\C[x_1,\dots,x_n]$, the ring of polynomials in commuting variables, then its field of fractions is given by the field of rational functions.

In the non-commutative situation, though the construction is highly non-trivial, there exists also a skew field of fractions of $\C\left\langle x_{1},\dots,x_{n}\right\rangle $, which is uniquely determined by some universal property; it is denoted by $\C\plangle x_{1},\dots,x_{n}\prangle$, and sometimes it is also simply called the \emph{free (skew) field}. An element in this field is called a \emph{rational function}, as it can be obtained by taking repeatedly sums, products and inverses from polynomials. We won't go further into the details of the construction of $\C\plangle x_{1},\dots,x_{n}\prangle$, but we collect some basic facts in the following theorem. For interested readers, all these statements can be found in \cite[Chapter 7]{Coh06}.

\begin{theorem}\label{thm:free field}
There exists the field of fractions $\C\plangle x_{1},\dots,x_{n}\prangle$ of $\C\left\langle x_{1},\dots,x_{n}\right\rangle $ such that the inner rank of a matrix over $\C\left\langle x_{1},\dots,x_{n}\right\rangle $ stays the same if this matrix is considered as a matrix over $\C\plangle x_{1},\dots,x_{n}\prangle$.

Moreover, any square matrix over $\C\plangle x_{1},\dots,x_{n}\prangle$ is full if and only if it is invertible. In particular, a full matrix over $\C\left\langle x_{1},\dots,x_{n}\right\rangle $ is invertible over $\C\plangle x_{1},\dots,x_{n}\prangle$.
\end{theorem}

Therefore, in the following, for a matrix $A$ over $\C\left\langle x_{1},\dots,x_{n}\right\rangle $, we do not need to distinguish between its inner rank over polynomials and its inner rank over rational functions; this common inner rank is denoted by $\rho\left(A\right)$.

Recall that in Section \ref{sec:inner rank}, we introduced the inner rank as an analogue of the rank for scalar-valued matrices; now with the existence of the free field, we have the equivalence between fullness and non-singularity, which stresses even more the analogy to the case of scalar-valued matrices.

Therefore, for a given full matrix $A$ over $\C\left\langle x_{1},\dots,x_{n}\right\rangle $, each entry in $A^{-1}$ is clearly a rational function. Actually, for any row vector $u$ and any column vector $v$ over $\C$, $uA^{-1}v$ is also a rational function because it is just a linear combination of some rational functions. What might not be easy to see is that the converse also holds, namely, any rational function $r$ can be written in the form $r=uA^{-1}v$, for some full matrix $A$ over polynomials and two scalar-valued vectors $u$ and $v$; of course, the size of this matrix depends on the considered rational function. Moreover, this matrix $A$ can be chosen to be linear, though the size of $A$ may increase for exchange. This culminates in the following definition from \cite{CR99}.

\begin{definition}\label{def:linear representation}
Let $r$ be a rational function. A \emph{linear representation of} $r$ is a tuple $\rho=(u,A,v)$ consisting of a linear full matrix $A\in M_{k}(\C\left\langle x_{1},\dots,x_{n}\right\rangle )$ (for some $k\in\N$), a row vector $u\in M_{1,k}(\C)$ and a column vector $v\in M_{k,1}(\C)$ such that $r=uA^{-1}v$.
\end{definition}

In \cite{CR99}, such linear representations were used to give an alternative construction of the free field. That indeed each element in the free field admits a linear representation is a fundamental result, which is a direct consequence of the approach of \cite{CR99}, but follows also from the general theory presented in \cite{Coh06}; see also \cite{Vol18}.

\begin{theorem}\label{thm:linear representation}
Each rational function $r \in \C\plangle x_1,\dots,x_n\prangle$ admits a linear representation in the sense of Definition \ref{def:linear representation}.
\end{theorem}

The idea of realizing noncommutative rational functions by inverses of linear matrices has been known for more than fifty years; and was rediscovered several times in many different branches of mathematics as well as computer science and engineering.

For the special case of noncommutative polynomials it was introduced, under the name ``linearization trick'', to the community of free probability by the work of Haagerup and Thorbj\o rnsen \cite{HT05} and Haagerup, Schultz, and Thorbj\o rnsen \cite{HST06}, building on earlier operator space versions; for the latter see in particular the work of Pisier \cite{Pis18}. Similar concepts were developed by Anderson \cite{And12,And13} and were used in \cite{BMS17} in order to study evaluations of noncommutative polynomials in noncommutative random variables by means of operator-valued free probability theory. Later, in \cite{HMS18}, these methods were generalized also to noncommutative rational expressions.

\subsection{Evaluation of rational functions}\label{sec:evaluation}

Let $X=(X_{1},\dots,X_{n})$ be a tuple of elements in a unital algebra $\A$, then its evaluation map $\ev_{X}$ from $\C\left\langle x_{1},\dots,x_{n}\right\rangle $ to $\A$ is well-defined as a homomorphism. The question we want to address in this section is how can we define the evaluation for rational functions. Unfortunately, the evaluation cannot be well-defined for all algebras without additional assumptions. Here is an example which illustrates the problem: considering $\A=B\left(H\right)$ for some infinite dimensional separable Hilbert space, let $l$ denote the one-sided left-shift operator for some basis of $H$, indexed by natural numbers; then $l^{\ast}$ is the right-shift operator and we have $l\cdot l^{\ast}=1$ but $l^{\ast}\cdot l\neq1$; so it is clear that the evaluation of $y\left(xy\right)^{-1}x$ at $(l,l^{\ast})$ is $l^{\ast}(l\cdot l^{\ast})l=l^{\ast}l\neq1$; however, since we have $y(xy)^{-1}x=1$ as rational function, there is no consistent way to define the evaluation of this function for the arguments $l$ and $l^*$. So it's natural to consider algebras in which a left inverse is also a right inverse to avoid such a problem; actually, we require algebras to be stably finite (see Definition \ref{def:stably finite}) in order to make sure that we have a well-defined evaluation.

\begin{theorem}
\label{thm:evaluation}Let $\A$ be a stably finite algebra, then for any rational function $r$ in the free field $\C\plangle x_{1},\dots,x_{n}\prangle$, we have a well-defined $\A$-domain $\dom_{\A}(r)\subseteq\A^{n}$ and an evaluation $r(X)$ for any $X\in\dom_{\A}(r)$.
\end{theorem}

Actually, the converse also holds in some sense, see Theorem 7.8.3 in the book \cite{Coh06} (we should warn the reader that the terminology used there is quite different from ours). See also \cite[Theorem 6.1]{HMS18}, where rational functions are treated as equivalence classes of formal rational expressions evaluated on matrices of all sizes. For the reader's convenience, here we give a proof of Theorem \ref{thm:evaluation} by using linear representations.

\begin{definition}
\label{def:evaluation}For a linear representation $\rho=(u,A,v)$, we define its $\A$\emph{-domain}
\[
\dom_{\A}(\rho):=\{X\in\A^{n}\bigm|A(X)\text{ is invertible as a matrix over }\A\};
\]
and for a given rational function $r$, we define its $\A$\emph{-domain}
\[
\dom_{\A}(r):=\bigcup_{\rho}\dom_{\A}(\rho),
\]
where the union is taken over all possible linear representations of $r$. Then we define the \emph{evaluation} of $r$ at a tuple $X\in\dom_{\A}(r)$ by
\[\Ev_{X}(r):=uA(X)^{-1}v\]
for any linear representation $\rho=(u,A,v)$ satisfying $X\in\dom_{\A}(\rho)$. We also denotes this evaluation by $r(X)$.
\end{definition}

Of course, as the choice of the linear representations for a rational function is not unique, we have to prove that different choices always give the same evaluation.

\begin{proof}[Proof of Theorem \ref{thm:evaluation}]
Let $\rho_{1}=(u_{1},A_{1},v_{1})$ and $\rho_{2}=(u_{2},A_{2},v_{2})$ be two linear representations of a rational function $r$, so that
\[r=u_{1}A_{1}^{-1}v_{1}=u_{2}A_{2}^{-1}v_{2}\]
in $\C\plangle x_{1},\dots,x_{n}\prangle$. We need to prove that for any $X\in\dom_{\A}(\rho_{1})\cap\dom_{\A}(\rho_{2})$, we have $u_{1}A_{1}(X)^{-1}v_{1}=u_{2}A_{2}(X)^{-1}v_{2}$. It is not difficult to verify that the tuple
\[\left(\begin{pmatrix}u_{1} & u_{2}\end{pmatrix},\begin{pmatrix}A_{1} & \0\\\0 & -A_{2}\end{pmatrix},\begin{pmatrix}v_{1}\\v_{2}\end{pmatrix}\right)\]
is a linear representation of zero in the free field with $\A$-domain equal to $\dom_{\A}(\rho_{1})\cap\dom_{\A}(\rho_{2})$; hence it suffices to prove that, for any linear representation $\rho=(u,A,v)$ of zero, we have $uA(X)^{-1}v=0$ for any $X\in\dom_{\A}(A)$.

Now suppose that $uA(X)^{-1}v\neq0$ for some linear representation $\rho=(u,A,v)$ of the zero function; then
\[\begin{pmatrix}0 & u\\v & A(X)\end{pmatrix}\in M_{k+1}(\A)\]
has inner rank $k+1$ over $\A$ by Proposition \ref{prop:invertible minor} since $\A$ is stably finite. On the other hand, since the free field $\C\plangle x_{1},\dots,x_{n}\prangle$ is stably finite, we can apply again Proposition \ref{prop:invertible minor} to show that
\[L=\begin{pmatrix}0 & u\\v & A\end{pmatrix}\in M_{k+1}(\C\left\langle x_{1},\dots,x_{n}\right\rangle )\]
has inner rank $k$ as $uA^{-1}v=0$ over $\C\plangle x_{1},\dots,x_{n}\prangle$; and thus $L$ has a rank factorization over $\C\left\langle x_{1},\dots,x_{n}\right\rangle $, which leads to the same factorization over $\A$ for $L(X)$ because the evaluation of polynomials is always well-defined as a homomorphism. So we can conclude that the inner rank of $L(X)$ over $\A$ is at most $k$, which is a contradiction.
\end{proof}

We close this subsection by remarking that this definition of evaluation is consistent with the usual notion of evaluation. That is, given any polynomial $p$, in order to see that the above definition coincides with the usual one, we should find a linear representation $\rho=(u,A,v)$ such that $uA(X)^{-1}v$ equals $p(X)$, the usual evaluation of $p$ at $X$, for any $X\in\A^{n}$. Such a linear representation can actually be constructed by following the rules in some algorithm of constructing linear representations; see, for example, \cite[Algorithm 5.3]{HMS18}. From there one can also see that the arithmetic operations between rational functions in the free field clearly give the corresponding arithmetic operations between their evaluations by linear representations.

\subsection{\label{subsec:rational closure}Rational closure and division closure}

In this subsection, we introduce some other constructions which are closely related to rational functions. In particular, the notion of a rational closure is needed to consider the strong Atiyah property in the next section.

\begin{definition}\label{def:rational closure}
Let $\phi:\C\left\langle x_{1},\dots,x_{n}\right\rangle \rightarrow\A$ be a homomorphism into a unital complex algebra $\A$ and let us denote by $\Sigma_{\phi}$ the set of all matrices whose images are invertible under the matricial amplifications of $\phi$, i.e.,
\[\Sigma_{\phi}=\bigcup_{k=1}^{\infty}\{A\in M_{k}(\C\left\langle x_{1},\dots,x_{n}\right\rangle )\bigm|\phi^{(k)}(A)\text{ is invertible in }M_{k}(\A)\},\]
where $\phi^{(k)}$ is given by applying $\phi$ entrywisely to any matrix in $M_{k}(\C\left\langle x_{1},\dots,x_{n}\right\rangle)$.
The \emph{rational closure of} $\C\left\langle x_{1},\dots,x_{n}\right\rangle $ with respect to $\phi$, denoted by $\mathcal{R}_{\phi}$, is the set of all entries of inverses of matrices in the image of $\Sigma_{\phi}$.
\end{definition}

We actually will only consider, in the next section, the case when $\phi$ is given by the evaluation of some tuple of elements in $\A$. Nonetheless, in this subsection, we introduce some basic facts about this construction in general.

\begin{lemma}
(See \cite[Proposition 7.1.1 and Theorem 7.1.2]{Coh06}) The rational closure $\mathcal{R}_{\phi}$ for any given homomorphism $\phi$ is a subalgebra of $\A$ containing the image of $\C\left\langle x_{1},\dots,x_{n}\right\rangle $.
\end{lemma}

Unlike the rational functions, the rational closure does not need to be a division ring in general. But it has a nice property concerning inverses: if an element $r\in\mathcal{R}_{\phi}$ is invertible in $\A$, then $r^{-1}\in\mathcal{R}_{\phi}$. In other words, the rational closure is closed under taking inverses in $\A$. A closely related notion is the smallest subalgebra that has this property.

\begin{definition}
\label{def:division closure}Let $\B$ be a subalgebra of $\A$. The \emph{division closure of $\B$} in $\A$, denoted by $\C\plangle\B\prangle$,  is the smallest subalgebra of $\A$ containing $\B$ which is closed under taking inverses in $\A$, i.e., if $d\in\D$ is invertible in $\A$, then $d^{-1}\in\D$. In particular, if $\B=\C\left<X_1,\cdots,X_n\right>$, i.e., $\B$ is generated by some elements $X_1,\cdots,X_n$ in $\A$, then we denote the division algebra of $\B$ by $\C\plangle X_1,\cdots,X_n\prangle$.
\end{definition}

From the definition it follows that the rational closure $\mathcal{R}_{\phi}$ for any homomorphism $\phi$ from $\mathbb{C}\left\langle x_{1},\dots,x_{n}\right\rangle$ to a unital algebra $\A$ always contains the division closure $\D$ of $\phi(\mathbb{C}\left\langle x_{1},\dots,x_{n}\right\rangle)$. However, in some nice situation, these two algebras coincide with each other.

\begin{proposition}
\label{prop:division closure}Let $\phi$ be a homomorphism from $\mathbb{C}\left\langle x_{1},\dots,x_{n}\right\rangle$ to a unital algebra $\A$ and $\D$ is the division closure of $\phi(\mathbb{C}\left\langle x_{1},\dots,x_{n}\right\rangle)$. If one of the rational closure $\mathcal{R}_{\phi}$ or the division closure $\D$ is a division ring, then $\mathcal{R}_{\phi}=\D$.
\end{proposition}

\begin{proof}
If $\mathcal{R}_{\phi}$ is a division ring, then $\D$ is also a division ring since $\D\subseteq\mathcal{R}_{\phi}$ and $\D$ is closed under taking inverses. Therefore, it suffices to show that if $\D$ is a division ring, then $\mathcal{R}_{\phi}=\D$. Hence, by the definition of rational closure, we need to show that for any matrix $A$ over $\mathbb{C}\left\langle x_{1},\dots,x_{n}\right\rangle$ whose image $\phi(A)$ is invertible over $\A$, $\phi(A)^{-1}$ is actually a matrix over $\D$. For that purpose, we want to prove for any $n\in\N$, the algebra $M_n(\D)$ is closed under taking inverses in $M_n(\A)$. Therefore, the proof will be completed by the following lemma.
\end{proof}

\begin{lemma}
If a division closure $\D$ is a division ring, then for any $n\in\N$, $M_n(\D)$ is closed under taking inverses.
\end{lemma}

The proof of this lemma can be done by induction in the same way as for Lemma 2.3 in \cite{Yin18}, which shows that some recursively defined ring (which we should replace here by the division closure $\D$) has the same closure property as stated in this lemma. An alternative proof of Proposition \ref{prop:division closure} can also be found in \cite{Rei06}, relying on the notion of a von Neumann regular ring.

Of course, the question when the rational closure $\mathcal{R}_{\phi}$ or division closure $\D$ becomes a division ring is in general not an easy one. We will discuss in Section \ref{subsec:Atiyah} a situation where this happens .

We close this section by a lemma on the diagonalization of matrices over a division ring; this will play an important role in the next section.

\begin{lemma}
\label{lem:diagonalization}Suppose that $\mathcal{K}$ is a division ring and let $A$ be a matrix over $\mathcal{K}$. Then there exist two invertible matrices $U$ and $V$ over $\mathcal{K}$ such that
\[
UAV=\begin{pmatrix}\1_r & \0\\\0 & \0\end{pmatrix}
\]
where the size of the block $\1_{r}$ is equal to $\rho_{\mathcal{K}}(A)$, the inner rank of $A$ over $\mathcal{K}$. In particular, a full matrix over $\mathcal{K}$ is invertible.
\end{lemma}

\begin{proof}
Though the entries in the matrix $A$ are noncommutative, the Gaussian elimination method still works here, as every non-zero element in a division ring is invertible. By permutations of rows and columns, we can write any full matrix $A$ in the form
\[
A=\begin{pmatrix}a & c\\d & B\end{pmatrix}\,
\]
with $a\neq0$ (since $A\neq \0$), then with the invertible matrices
\[
U=\begin{pmatrix}1 & \0 \\-da^{-1} & \1_{k-1}\end{pmatrix},\ V=\begin{pmatrix}1 & -a^{-1}c \\\0 & \1_{k-1}\end{pmatrix},
\]
we have
\[
UAV=\begin{pmatrix}1 & \0 \\\0 & B-da^{-1}c\end{pmatrix},
\]
where the block $B-da^{-1}c$ has to be a full matrix as $UAV$ does. Therefore, by induction on the dimension, we can see that each full matrix is invertible.

Now, given a (not necessarily full) matrix $A$, we have invertible matrices $U$ and $V$ such that
\[
UAV=\begin{pmatrix}\1_{r} & \0\\\0 & \0\end{pmatrix}
\]
by the Gaussian elimination. Then it remains to show $r=\rho_{\mathcal{K}}(A)$, that is, to show
\[
\rho_{\mathcal{K}}\begin{pmatrix}1_{r} & \0 \\\0 & \0\end{pmatrix}=r.
\]
This is achieved by applying Theorem \ref{thm:full minor}, provided that the assumptions in this theorem hold. So what we need to check is that the set of full square matrices is closed under products and diagonal sums, which can be inferred from the fact that full matrices are invertible. This fact has already been proved in the first paragraph and so the proof is completed.
\end{proof}

\section{Affiliated unbounded operators and Atiyah property}\label{sec:affiliated_operators_and_Atiyah}

In this section we will finally apply noncommutative rational functions to our tuple of operators from a tracial $W^*$-probability space. In order to get a reasonable theory we will allow to invert our operators as unbounded operators, affiliated to our von Neumann algebra. The linearization trick will then allow us to transfer our results about the invertibility from matrices over polynomials to rational functions.
Questions about comparing an analytic notion of the rank with our algebraic notion of the inner rank will play a crucial role in this context. This will lead naturally to the notion of the ``strong Atiyah property'', which we investigate in Section \ref{subsec:Atiyah}.

\subsection{Realization of free field}\label{subsec:free field realization}

Let $\left(\M,\tau\right)$ be a tracial $W^{\ast}$-probability space as before. In this section, we denote by $\A$ the set of all closed and densely defined linear operators affiliated with $\M$, which is known to be a $\ast$-algebra containing $\M$. An important and well-known fact is that the polar decomposition also holds in this case; see, for example, Sect. 9.29 in \cite{SZ}.

\begin{lemma}
Let $X$ be a closed and densely defined operator on some Hilbert space $H$, then we have $X=U\left|X\right|$, where $\left|X\right|=(X^{\ast}X)^{{1}/{2}}$ is a positive selfadjoint (so necessarily closed and densely defined) operator and $U$ is a partial isometry such that $U^*U=p_{(\ker(X))^{\bot}}$ and $UU^*=p_{\overline{\im(X)}}$. Moreover, $X$ is affiliated with $\M$ if and only if $U\in\M$ and $\left|X\right|$ is affiliated with $\M$.
\end{lemma}

Since we are in the finite setting, i.e., $\tau$ is a trace on $\M$, we have in this situation that 
$$\tau(p_{(\ker(X))^{\bot}})=\tau(U^*U)=\tau(UU^*)=\tau(p_{\overline{\im(X)}}).$$
Let us record this in the following lemma, which can be seen as the unbounded analogue of Lemma \ref{lem:kernels}.

\begin{lemma}\label{lem:kernels unbdd}
Given $X\in\A$, let $p_{\ker(X)}$ and $p_{\overline{\im(X)}}$ denote the orthogonal projections onto $\ker(X)$ and the closure of $\im(X)$, respectively. Then they belong both to $\M$ and satisfy
\[\tau(p_{\ker(X)})+\tau(p_{\overline{\im(X)}})=1.\]
\end{lemma}

A crucial consequence of this is the following: consider an operator $X\in\A$ with $\ker(X)=\{0\}$; then $p_{\ker(X)}=0$, and thus, by the lemma above, $\tau(p_{\overline{\im(X)}})=1$; which implies, by the faithfulness of $\tau$, that $p_{\overline{\im(X)}}=1$; which means that the image of $X$ is dense in $H$; and hence the inverse of $X$ exists as an unbounded operator. The conclusion of all this is that in our tracial $W^*$-setting the invertibility of an affiliated unbounded operator is, as in the finite-dimensional case, just the question whether the operator is injective or not; or to put it in a more algebraic way, an operator is invertible if and only if it has no non-trivial zero divisor.

Moreover, for each integer $N$, we can consider the matricial extension $(M_{N}(\M),\tr_{N}\circ\tau^{(N)})$, where $\tr_{N}$ is the normalized trace on $M_N(\C)$ and $\tau^{(N)}$ the matricial amplification of $\tau$ by applying $\tau$ entrywisely. Then $(M_{N}(\M),\tr_{N}\circ\tau^{(N)})$ is again a $W^{\ast}$-probability space and we have the matricial extension $M_{N}(\A)$, which is the $\ast$-algebra of closed and densely defined linear operators affiliated to $(M_{N}(\M),\tr_{N}\circ\tau^{(N)})$.
With the help of the polar decomposition and its matricial extended version, we can see that $\A$ is stably finite. This statement is probably well-known to experts; however, since we could not localize a proof in the literature, we provide here a short proof. The fact, that the statement as well as the proof can be extended from the bounded to the unbounded case was brought to our attention by Dima Shlyakhtenko, see \cite{HMS18}.

\begin{lemma}
$\A$ is stably finite and thus the evaluation of rational functions is well-defined on $\A$.
\end{lemma}

\begin{proof}
Since $M_{N}(\A)$ is, for each integer $N$, also a $\ast$-algebra of affiliated operators to a $W^*$-probability space, it suffices to treat the case $N=1$, i.e., to prove that for any $X,Y\in\A$ the equation $XY=1$ implies also $YX=1$. First, let $X,Y\in\A$ with $XY=1$ and $X=X^{\ast}$. Then $Y^{\ast}X=1$, hence $Y^{\ast}=Y^{\ast}XY=Y$, and thus $YX=1$. Next, we consider now arbitrary $X,Y\in\A$ with $XY=1$. By the polar decomposition, we can write $Y=U\left|Y\right|$ with a partial isometry $U\in\M$ and $\left|Y\right|\in\A$. Note that $\left|Y\right|$ is selfadjoint and satisfies $XU\left|Y\right|=XY=1$ in $\A$, so by the previous argument we have $\left|Y\right|XU=1$. Then $U$ is injective and so it must be unitary by the previous lemma. Hence $Y$ is also injective and thus invertible with inverse $X$.

That the evaluation of rational functions is well-defined on $\A$ follows then from Theorem \ref{thm:evaluation}.
\end{proof}


\begin{definition}\label{def:rank}
For any $P\in M_{N}(\A)$, we define its \emph{rank} as
\[\rank(P)=\Tr_{N}\circ\tau^{(N)}(p_{\overline{\im(P)}}),\]
where $\Tr_N$ is the unnormalized trace on $M_N(\C)$.
\end{definition}

For later use, we point out that, due to Lemma \ref{lem:kernels unbdd}, it holds true that
\begin{equation}\label{eq:rank_kernel}
\rank(P) = N - \Tr_{N}\circ\tau^{(N)}(p_{\ker(P)}).
\end{equation}

With this analytic notion of a rank we can rephrase the statement on the invertibility of affiliated operators as follows.

\begin{lemma}
\label{invertible and rank}$P\in M_{N}(\A)$ is invertible if and only if $\rank(P)=N$.
\end{lemma}

Compare this to the corresponding algebraic statement (namely, Theorem \ref{thm:free field}):
$P\in M_N(\C\left\langle x_{1},\dots,x_{n}\right\rangle)$ is invertible (in matrices over the free field) if and only if its inner rank $\rho(P)=N$.

Recall now that Theorem \ref{thm:maximality of Delta} asserts that if a tuple $X=(X_{1},\dots,X_{n})$ of random variables satisfies $\Delta(X)=n$, then a linear full matrix $P$ over $\C\left\langle x_{1},\dots,x_{n}\right\rangle$ gives an invertible evaluation $P(X)$ over $\A$. Or equivalently, in terms of our notions of ranks: for a linear matrix $P$ in $M_{N}(\C\left\langle x_{1},\dots,x_{n}\right\rangle)$,
\[
\rho(P)=N\implies\rank(P(X))=N.
\]
One of the main goals of this section is to show that actually these two ranks are equal when $\Delta(X)=n$; actually, we will give in the following theorem a couple of equivalent characterizations of the equality of the algebraic and the analytic rank.

\begin{theorem}
\label{thm:Atiyah-1}Let $(\M,\tau)$ be a tracial $W^{\ast}$-probability space and $\A$ the $\ast$-algebra of affiliated unbounded operators.
For a given tuple $X=(X_{1},\dots,X_{n})$ in $\M^n$, we consider the evaluation map $\ev_{X}:\C\left\langle x_{1},\dots,x_{n}\right\rangle \rightarrow\A$. Then the following statements are equivalent.
\begin{enumerate}
\item\label{it:Atiyah-1 i} For any $N\in\N$ and $P\in M_{N}(\C\left\langle x_{1},\dots,x_{n}\right\rangle )$ we have: if $P$ is linear and full, then $P(X)\in M_{N}(\A)$ is invertible.
\item\label{it:Atiyah-1_ii} For any $N\in\N$ and $P\in M_{N}(\C\left\langle x_{1},\dots,x_{n}\right\rangle )$ we have: if $P$ is full, then $P(X)\in M_{N}(\A)$ is invertible.
\item\label{it:Atiyah-1_iii} For any $N\in\N$ and $P\in M_{N}(\C\left\langle x_{1},\dots,x_{n}\right\rangle )$ we have: $\rank(P(X))=\rho(P)$.
\item\label{it:Atiyah-1_iv} We have $X\in\dom_\A(r)$ for each $r\in \C\plangle x_{1},\dots,x_{n}\prangle$ and $\Ev_{X}$ as introduced in Definition \ref{def:evaluation} induces an injective homomorphism $\Ev_{X}: \C\plangle x_{1},\dots,x_{n}\prangle\rightarrow\A$ that extends the evaluation map $\ev_{X}:\C\left\langle x_{1},\dots,x_{n}\right\rangle \rightarrow\A$.
\end{enumerate}
Moreover, if the equivalent conditions above are satisfied, then
\begin{equation}\label{eq:rank_equality}
\rank(P(X)) = \rho(P) = \rho_\A(P(X)) \qquad\text{for all $P\in M_{N}(\C\left\langle x_{1},\dots,x_{n}\right\rangle )$},
\end{equation}
where $\rho_{\A}(P(X))$ denotes the inner rank of $P(X)$ over the algebra $\A$.
\end{theorem}

Before giving the proof of Theorem \ref{thm:Atiyah-1}, some comments are in order.

\begin{remark}\label{rem:zero_divisors}
The property formulated in Item \ref{it:Atiyah-1_iv} of Theorem \ref{thm:Atiyah-1} implies that $r(X)$, for any non-zero rational function $r$, is not a (left) zero divisor over $\M$, i.e., there is no $0\neq w\in\M$ such that $r(X)w=0$ holds. Indeed, since $\Ev_{X}$ is a homomorphism defined on the whole free field, we see that $r(X)$ is invertible in $\A$ with $r(X)^{-1} = \Ev_X(r^{-1})$; thus, for any $w\in\M$ satisfying $r(X)w=0$, we may read the latter as an equation in $\A$ and so derive by multiplication from the left with $r(X)^{-1}$ that necessarily $w=0$, as asserted.

Similarly, the property formulated in Item \ref{it:Atiyah-1_ii} of Theorem \ref{thm:Atiyah-1} excludes that $P(X)$, for every full $N\times N$ matrix $P$ over $\C\langle x_1,\dots,x_n\rangle$, is a (left) zero divisor in $M_N(\M)$.
\end{remark}

\begin{remark}\label{rem:free_field_involution}
The anti-linear involution $\ast$ on $\C\langle x_1,\dots,x_n\rangle$, which is determined by $1^\ast=1$ and $x_j^\ast = x_j$ for $j=1,\dots,n$, extends uniquely to an anti-linear involution $\ast$ on the free field $\C\plangle x_1,\dots,x_n\prangle$. On the level of linear representations, this can be expressed very nicely: if $\rho=(u,A,v)$ is a linear representation of $r$, then $\rho^\ast := (v^\ast,A^\ast,u^\ast)$ gives a linear representation of $r^\ast$; note that with $A$ also $A^\ast$ is full in $M_N(\C\langle x_1,\dots,x_n\rangle)$. Therefore, in the case where our tuple $X$ consists of selfadjoint operators, the homomorphism $\Ev_{X}: \C\plangle x_{1},\dots,x_{n}\prangle\rightarrow\A$ considered in Item \ref{it:Atiyah-1_iv} of Theorem \ref{thm:Atiyah-1} is automatically a $\ast$-homomorphism whenever it exists.
\end{remark}

For the proof of Theorem \ref{thm:Atiyah-1} we will also need the fact that the analytic rank does not change when multiplied by invertible matrices over $\A$.

\begin{lemma}
\label{inverstible matrix preserve rank}(See \cite[Lemma 2.3]{Lin93}) If $Q$ is invertible in $M_{N}(\A)$, then $\rank(P)=\rank(PQ)=\rank(QP)$ for any $P\in M_{N}(\A)$.
\end{lemma}

\begin{proof}[Proof of Theorem \ref{thm:Atiyah-1}]
It's easy to see that (ii)$\implies$(i) is trivial and (iii)$\implies$(ii) follows from Lemma \ref{invertible and rank}.

(iv)$\implies$(iii): Assume that $P\in M_{N}(\C\left\langle x_{1},\dots,x_{n}\right\rangle )$ has inner rank $\rho(P)=r$, then by Lemma \ref{lem:diagonalization}, there exist two invertible matrices $U$ and $V$ over the free field $\C(\plangle x_{1},\dots,x_{n}\prangle)$ such that
\[
UPV=\begin{pmatrix}\1_r & \0\\\0 & \0\end{pmatrix}.
\]
As we assume the statement (iv), the extended evaluation $\Ev_{X}:\C\plangle x_{1},\dots,x_{n}\prangle\rightarrow \A$, as a homomorphism, implies that $U(X),V(X)$ are invertible. By Lemma \ref{inverstible matrix preserve rank}, we obtain
\[
\rank(P(X))=\rank(U(X)P(X)V(X))=\rank\begin{pmatrix}\1_{r} & \0\\\0 & \0\end{pmatrix}=r.
\]

(i)$\implies$(iv): First, recall from Definition \ref{def:evaluation} that a rational function $r$ in the free field $\C\plangle x_{1},\dots,x_{n}\prangle$ satisfies $X\in\dom_\A(r)$ if there is a linear representation $\rho=(u,A,v)$ of $r$ such that $X\in \dom_\A(\rho)$, i.e., $A(X)$ is invertible; but in fact, each linear representation of $r$ (whose existence is guaranteed by Theorem \ref{thm:linear representation}) has this property due to our assumption (i) as $A$ is full.
Thus, according to Definition \ref{def:evaluation} and Theorem \ref{thm:evaluation}, the evaluation $\Ev_X(r)$ is well-defined for each $r\in \C\plangle x_{1},\dots,x_{n}\prangle$ and thus induces a map $\Ev_X: \C\plangle x_1,\dots,x_n\prangle \to \A$. Moreover, we can infer from the proof of Theorem \ref{thm:evaluation} that the evaluation of rational functions via linear representations respects the arithmetic operations between rational functions. Hence the evaluation map $\Ev_X: \C\plangle x_1,\dots,x_n\prangle \to \A$ forms a homomorphism which agrees with $\ev_{X}$ on $\C\left\langle x_{1},\dots,x_{n}\right\rangle$. Moreover, $\Ev_{X}$ has to be injective as a homomorphism from a skew field.

Suppose now that the equivalent conditions are satisfied. Then, for any matrix $P\in M_{N}(\C\left\langle x_{1},\dots,x_{n}\right\rangle )$, we can consider the rank factorization $P(X)=AB$ of $P(X)$ over $\A$, where $A\in M_{N,r}(\A)$ and $B\in M_{r,N}(\A)$ for $r := \rho_\A(P(X)) \leq N$. This can be rewritten as $P(X) = \hat{A} \hat{B}$ with the square matrices $\hat{A},\hat{B} \in M_N(\A)$ that are defined by
\[
\hat{A} := \begin{pmatrix} A & \0_{N \times (N-r)} \end{pmatrix} \qquad\text{and}\qquad \hat{B} := \begin{pmatrix} B\\ \0_{(N-r) \times N} \end{pmatrix}.
\]
From this, we see that $(\tr_N \circ \tau^{(N)})(p_{\ker(\hat{A})}) \geq \frac{N-r}{N}$, so that $\rank(\hat{A}) \leq r$ by \eqref{eq:rank_kernel}; thus, since $\im(P(X))\subset\im(\hat{A})$, it follows that
\[\rank(P(X)) \leq \rank(\hat{A})\leq r.\]
On the other hand, we may observe that in general
\[r = \rho_{\A}(P(X)) \leq \rho(P),\]
because each rank factorization of $P$ yields after evaluation at $X$ a factorization of $P(X)$ over $\A$. Finally, the third property in the theorem gives us
\[\rho(P) = \rank(P(X)).\]
Thus, in summary, the asserted equality \eqref{eq:rank_equality} follows.
\end{proof}

\begin{remark}
The property in Theorem \ref{thm:Atiyah-1} \ref{it:Atiyah-1_iv} also says that the image of the free field under the evaluation map forms a division subring of $\A$ that contains the algebra $\C\left\langle X_{1},\dots,X_{n}\right\rangle$ generated by $X_1,\dots, X_n$. Therefore, we can also infer that the division closure $\D$ of $\C\left\langle X_{1},\dots,X_{n}\right\rangle$ is contained in the image of the free field $\Ev_{X}(\C\plangle x_{1},\dots,x_{n}\prangle)$; so the division closure is in this case a division ring. Such a result was first established by Linnell in his paper \cite{Lin93} for free groups. Additionally, then $\D$ is actually the rational closure $\mathcal{R}$ of $\C\langle X_{1},\dots,X_{n}\rangle$ with respect to $\ev_{X}$ by Proposition \ref{prop:division closure}. Combining this with the fact that $\D\subset\Ev_{X}(\C\plangle x_{1},\dots,x_{n}\prangle)\subset\mathcal{R}$, we have the following corollary.
\begin{quote}
If $X=(X_{1},\dots,X_{n})$ is a tuple of random variables in some tracial $W^{\ast}$-probability space $(\M,\tau)$ such that $X$ generates the free field, i.e., $X$ satisfies the equivalent properties in Theorem \ref{thm:Atiyah-1}, then the division closure and the rational closure of the algebra $\C\langle X_{1},\dots,X_{n}\rangle$ in $\A$ are equal and are isomorphic to the free field by $\Ev_X$.
\end{quote}
\end{remark}

\begin{remark}
Let us address the commutative counterpart of Theorem \ref{thm:Atiyah-1};  so let $X$ now be a tuple of commuting operators. In this case a rational function can be written as the quotient of two polynomials, and thus the question whether all rational functions are well-defined for the evaluation at $X$ reduces to the question if all non-zero polynomials are invertible after evaluation at $X$. 
The statements and the proof of Theorem \ref{thm:Atiyah-1} can quite easily be adapted to this much simpler commutative setting and
we have the following analogue of Theorem \ref{thm:Atiyah-1} for commuting operators. There we denote the algebra of commutative polynomials by $\C\left[x_{1},\dots,x_{n}\right]$ and its field of fractions by $\C\left(x_{1},\dots,x_{n}\right)$.

For a given tuple $X=(X_{1},\dots,X_{n})$ of commuting operators in a tracial $W^{\ast}$-probability space $(\M,\tau)$, the following statements are equivalent.
\begin{enumerate}
\item For any non-zero polynomial $p\in \C\left[x_{1},\dots,x_{n}\right]$ we have that $p(X)$ is invertible in the $\ast$-algebra $\A$ of unbounded operators affiliated to $\M$.
\item For any $N\in\N$ and $P\in M_{N}(\C\left[x_{1},\dots,x_{n}\right])$ we have: $\rank(P(X))=\rho(P)$, where $\rho(P)$ denotes the inner rank of $P$ over $\C\left(x_{1},\dots,x_{n}\right)$.
\item The evaluation map $\ev_X$ can be extended from $\C\left[x_1,\dots,x_n\right]$ to $\C\left(x_1,\dots,x_n\right)$.
\end{enumerate}

There are two significant differences between the commutative and the noncommutative case. First, to generate the field of fractions $\C\left(x_1,\dots,x_n\right)$ by a tuple $X$ of commuting operators, it is enough to ask all non-zero polynomials to be invertible (or equivalently, to have no zero divisors) after the evaluation at $X$. Whereas, in the noncommutative case, in Item (i) and (ii) of Theorem \ref{thm:Atiyah-1}, we need to take full matrices of all sizes into account rather than just $1$-dimensional full matrices. Secondly, the inner rank $\rho$ in Item (ii) cannot be taken over polynomials as in Item (iii) of Theorem \ref{thm:Atiyah-1}; it has to be taken over rational functions. This is because if there are more than two variables, the ring $\C\left[x_1,\dots,x_n\right]$ is not a Sylvester domain any more and so the inner rank over polynomials is not preserved when embedding into $\C\left(x_1,\dots,x_n\right)$ (see Section 5.5 and Section 7.5 in \cite{Coh06} for details).
\end{remark}

Note that in Theorem \ref{thm:maximality of Delta} we have proved that a tuple $X=(X_1,\dots,X_n)\in\M^n$ has the property $\Delta(X)=n$ if and only if Item \ref{it:Atiyah-1 i} in the Theorem \ref{thm:Atiyah-1} holds. Hence, in such a situation all the other properties formulated in Theorem \ref{thm:Atiyah-1} must also be satisfied. Thus we get the full set of equivalences from Theorem \ref{thm:intro-main}.

In particular, as explained in Remark \ref{rem:zero_divisors}, we obtain hereby a natural generalization to the case of rational functions of the result of \cite{CS16,MSW17} which said that for any non-zero $P\in\C\langle x_1,\dots,x_n\rangle$, $P(X)$ cannot be a zero divisor in $\M$.

Let us spell out these important consequences as the following corollary.

\begin{corollary}\label{cor:Delta_maximality_implies_Atiyah}
If $X=(X_{1},\dots,X_{n})$ is a tuple of random variables in some tracial $W^{\ast}$-probability space $(\M,\tau)$ with $\Delta(X)=n$, then
\begin{enumerate}
\item for any non-zero rational function $r\in \C\plangle x_1,\dots,x_n\prangle$, $r(X)$ is well-defined and invertible as an affiliated unbounded operator, and thus $r(X)$ is not zero and has no zero divisors;
\item for any $N\in\N$ and any $N\times N$ matrix $P$ over $\C\left\langle x_{1},\dots,x_{n}\right\rangle$, we have
\[\rank(P(X))=\rho(P),\] and in particular, if $P$ is full, then $P(X)$ is invertible over the affiliated unbounded operators and has no zero divisors.
\end{enumerate}
\end{corollary}

In \cite{CS16,MSW17}, $X$ is assumed to be a tuple of selfadjoint random variables, the attempt to exclude zero divisors among polynomial evaluations was motivated by the question whether atoms can appear in the analytic distributions of evaluations of non-constant selfadjoint polynomials. Recall that the \emph{analytic distribution $\mu_X$} of a normal random variable $X$ in $(\M,\tau)$ is defined as the unique Borel probability measure $\mu_X$ on $\C$ that satisfies $\int_\C z^k \overline{z}^l\, d\mu_X(t) = \tau(X^k (X^\ast)^l)$ for all $k,l\in\N_0$. In fact, now we can also exclude atoms for the analytic distribution of evaluations of non-constant rational functions $r$ that are selfadjoint in the sense that $r^\ast = r$ holds with respect to the involution $\ast$ introduced in Remark \ref{rem:free_field_involution}. We record this observation in the following corollary.

\begin{corollary}\label{cor:rational_atoms}
Let $X=(X_{1},\dots,X_{n})$ be a tuple of selfadjoint random variables in some tracial $W^{\ast}$-probability space $(\M,\tau)$ with the property that $\Delta(X)=n$. Then, for every non-constant selfadjoint rational function $r\in\C\plangle x_1,\dots,x_n\prangle$, the evaluation $r(X)$ is a selfadjoint unbounded operator affiliated to $\M$ and its analytic distribution $\mu_{r(X)}$ has no atoms.
\end{corollary}

Recall that for any selfadjoint unbounded operator $Y\in\A$, its \emph{analytic distribution $\mu_Y$} is defined as the Borel probability measure on $\R$ that is obtained by $\mu_Y = \tau \circ E_Y$, where $E_Y$ denotes the resolution of identity associated to $Y$; this is well-defined, since $E_Y$ takes its values in $\M$ as $Y$ is a selfadjoint operator affiliated to $\M$, and extends the concept of analytic distributions for bounded operators that we have used before.

\begin{proof}[Proof of Corollary \ref{cor:rational_atoms}]
Assume that $r$ is a non-constant selfadjoint rational function for which $\mu_{r(X)}$ has an atom, say at $s\in\R$. Then $p := p_{\ker(r(X)-s)} = E_{r(x)}(\{s\})$ is a projection in $\M$, which is non-zero as $\tau(p) = \mu_{r(X)}(\{s\}) \neq 0$ and satisfies $(r(X)-s)p=0$. This, however, contradicts Item (i) in Corollary \ref{cor:Delta_maximality_implies_Atiyah}, because $r-s$ is non-zero in $\C\plangle x_1,\dots,x_n\prangle$ as $r$ is non-constant.
\end{proof}

\subsection{Regularity of matrices with polynomial entries}\label{subsec:regularity_polynomials_matrices}

With Corollary \ref{cor:rational_atoms}, we generalized the results of \cite{CS16,MSW17} about absence of atoms from noncommutative polynomials to the case of noncommutative rational functions. This section is devoted to another possible extension of those results, namely from noncommutative polynomials to square matrices thereof. This is in the spirit of \cite{SS15}, but weakens their assumption that the involved noncommutative random variables are freely independent with non-atomic individual analytic distributions to the maximality of $\Delta$.

To be more precise, we will study the point spectrum of operators of the form $\X = P(X)$, where $P$ is square matrix over $\C\langle x_1,\dots,x_n\rangle$, say of size $N\times N$, and $X=(X_1,\dots,X_n)$ is a tuple of (not necessarily selfadjoint) operators in $(\M,\tau)$ with the property that $\Delta(X)=n$. For selfadjoint $\X$, this will allow us to locate the position and to specify the size of atoms in the analytic distribution $\mu_\X$ of $\X$ by purely algebraic quantities associated to $P$.

Recall that the point spectrum $\sigma_p(\X)$ of $\X$, seen as an operator in $B(L^2(\M,\tau))$, is defined as the set of all $\lambda\in\C$ for which $\ker(\X - \lambda \1_N) \neq \{0\}$. Alternatively, we can view $\X$ as an element in the $\ast$-algebra $M_N(\A)$ for the $\ast$-algebra of all closed and densely defined linear operators affiliated with $(\M,\tau)$; therefore, it is natural to consider the set $\sigma_\A(\X)$ of all $\lambda\in\C$ for which $\X - \lambda \1_N$ fails to be invertible in $M_N(\A)$. Using \eqref{eq:rank_kernel} and Lemma \ref{invertible and rank}, we see that both notions are equivalent; in fact, we have that
\begin{equation}\label{eq:analytic_spectra}
\sigma_p(\X) = \big\{\lambda \in \C \bigm| \rank(\X - \lambda \1_N) < N \big\} = \sigma_\A(\X).
\end{equation}

As we announced above, we show with the following theorem that, due to the assumption $\Delta(X)=n$, these ``analytic spectra'' $\sigma_p(\X) = \sigma_\A(\X)$ of $\X$ further agree with the ``algebraic spectrum'' $\sigma^\full(P)$ of $P$.

\begin{theorem}\label{thm:central eigenvalues and spectrum}
Suppose that $X=(X_1,\dots,X_n)$ is a tuple of elements in a tracial $W^\ast$-probability space $(\M,\tau)$ that satisfies $\Delta(X)=n$ and let any $P\in M_N(\C\langle x_1,\dots,x_n\rangle)$ be given. Then
\begin{equation}\label{eq:position of atoms}
\sigma^\full(P) = \sigma_p(P(X));
\end{equation}
moreover, we have that
\begin{equation}\label{eq:rank of atoms}
\rho(P - \lambda \1_N) = \rank(P(X) - \lambda \1_N) \qquad \text{for all $\lambda\in\C$}.
\end{equation}
\end{theorem}

Note that in general, for a matrix $P\in M_N(\C\langle x_1,\dots,x_n\rangle)$ and any tuple $X$ of operators in $\M$, we have $\sigma^\full(P) \subseteq \sigma^\full_\A(P(X)) \subseteq \sigma_\A(P(X))$, where $\sigma^\full_\A(P(X))$ is the set of all central eigenvalues (see Definition \ref{def:central eigenvalue}); these inclusions follow by the fact that $\rank(P(X)) \leq \rho_\A(P(X)) \leq \rho(P)$, which was established in the proof of Theorem \ref{thm:Atiyah-1}. Theorem \ref{thm:central eigenvalues and spectrum} thus provides a criterion for the equality of all these spectra and rank functions.

\begin{proof}[Proof of Theorem \ref{thm:central eigenvalues and spectrum}]
For every $\lambda\in\C$, we infer from Corollary \ref{cor:Delta_maximality_implies_Atiyah} by applying it to $P-\lambda\1_N$ that \eqref{eq:rank of atoms} holds. With the help of \eqref{eq:analytic_spectra}, we deduce from \eqref{eq:rank of atoms} that the sets $\sigma^\full(P)$ and $\sigma_\A(P(X))$ agree, which verifies the assertion \eqref{eq:position of atoms}.
\end{proof}

Theorem \ref{thm:central eigenvalues and spectrum} determines properties of certain analytic distributions in terms of purely algebraic conditions. Precisely, we have the following corollary.

\begin{corollary}\label{cor:atoms}
Let $X=(X_1,\dots,X_n)$ be a tuple of elements in a tracial $W^\ast$-probability space $(\M,\tau)$ that satisfies $\Delta(X)=n$ and let $P\in M_N(\C\langle x_1,\dots,x_n\rangle)$ be given such that the noncommutative random variable $\X := P(X)$ that lives in the tracial $W^\ast$-probability space $(M_N(\M),\tr_N \circ \tau^{(N)})$ is normal. Then the analytic distribution $\mu_\X$ of $\X$ has atoms precisely at the points in $\sigma^\full(P)$; in fact, we have that
\[
\mu_\X(\{\lambda\}) = \frac{1}{N}\big(N - \rho(P - \lambda \1_N)\big) \qquad \text{for each $\lambda\in\C$}.
\]
\end{corollary}

\begin{proof}
Since \eqref{eq:rank_kernel} yields that for every normal operator $\X$ in $M_N(\M)$
\[
\mu_\X(\{\lambda\}) = \tr_N\circ\tau^{(N)}\big(p_{\ker(\X-\lambda \1_N)}\big) = \frac{1}{N}\big(N - \rank(\X-\lambda \1_N)\big)
\]
holds for each $\lambda\in\C$ and hence $\sigma_p(\X) = \{\lambda\in\C \mid \mu_\X(\{\lambda\})>0 \}$, the assertion of the corollary is indeed an immediate consequence of Theorem \ref{thm:central eigenvalues and spectrum}.
\end{proof}

\begin{remark}\label{rem:atoms}
In the situation of Corollary \ref{cor:atoms}, if we suppose in addition that $\X$ is selfadjoint, we can extract some information about the non-microstates free entropy dimension $\delta^\ast(\X)$.
Recall that for one selfadjoint random variable $Y$ its free entropy dimension $\delta^\ast(Y)$ is determined by the sizes of all atoms of $\mu_Y$ (see also \eqref{eq:entropy_dimension_single_operator}).
Therefore, the formula \eqref{eq:rank of atoms} given in Theorem \ref{thm:central eigenvalues and spectrum} allows to express $\delta^\ast(\X)$ of the variable $\X = P(X)$ in terms of purely algebraic quantities associated to $P$; more precisely, we have that
$$\delta^\ast(\X) = 1 - \sum_{\lambda \in \sigma_\A(\X)} \mu_\X(\{\lambda\})^2 = 1 - \frac{1}{N^2} \sum_{\lambda \in \sigma^\full(P)} \big(N - \rho(P-\lambda\1_N)\big)^2.$$
Note that $\rho(P-\lambda\1_N)\in\N$, so that
\[
\delta^\ast(\X)\in\big\{\frac{k}{N^2} \bigm| k=0,\dots,N^2 \big\}.
\]
\end{remark}

On the other hand, Theorem \ref{thm:central eigenvalues and spectrum} tells us, loosely speaking, that each tuple $X=(X_1,\dots,X_n)$ of operators in a tracial $W^\ast$-probability space $(\M,\tau)$ satisfying the condition $\Delta(X)=n$ provides a model based on which the algebraic quantities $\rho(P)$ and $\sigma^\full(P)$ associated to matrices $P\in M_N(\C\langle x_1,\dots,x_n\rangle)$ can be computed by analytic means via the operator $P(X)$. We want to give an example for how this bridge can be used.

To be more precise, we can deduce Proposition \ref{prop:finite spectrum}, which is an intricate algebraic fact, from Theorem \ref{thm:central eigenvalues and spectrum}. Recall that Proposition \ref{prop:finite spectrum} says that the cardinality of $\sigma^\full(P)$ is at most $N$ for a matrix $P\in M_N(\C\langle x_1,\dots,x_n\rangle)$. So, if we have $\Delta(X)=n$ for a tuple $X=(X_1,\dots,X_n)$, then with the help of Theorem \ref{thm:central eigenvalues and spectrum}, Proposition \ref{prop:finite spectrum} follows once we have shown that the cardinality of $\sigma_p(P(X))$ is at most $N$; this is the content of the following proposition.

\begin{proposition}\label{prop:finite atoms}
Suppose that $X=(X_1,\dots,X_n)$ is a tuple of elements in a tracial $W^\ast$-probability space $(\M,\tau)$ that satisfies $\Delta(X)=n$ and let any $P\in M_N(\C\langle x_1,\dots,x_n\rangle)$ be given. Then $\sigma_\A(P(X))$ has at most $N$ elements.
\end{proposition}

In order to prove this we need the following lemma.

\begin{lemma}
Let $P$ be an operator in a tracial $W^*$-probability space $(\M,\tau)$. Assume that we have distinct eigenvalues $\lambda_1,\cdots,\lambda_k$ of $P$ on $L^2(\M,\tau)$. Denote by $p_i\in \M$ the orthogonal projection onto the eigenspace $\text{\rm kernel}(P-\lambda_i 1)$ of $\lambda_i$. Then we have
\[
\sum_{i=1}^k\tau(p_{i})=\tau(\bigvee_{i=1}^k p_{i}).
\]
\end{lemma}

\begin{proof}
This statement is probably well-known, but since we have not been able to localize it in the literature, we provide here a short proof. 

The statement is a consequence of the following standard result for projections $p$ and $q$ in finite von Neumann algebras (see, e.g., Lemma 2.2.3 in \cite{CDSZ17}): if we have $p\wedge q=0$, then 
$$\tau(p\vee p)=\tau(p)+\tau(p).$$
Since it is easy to check that in our case $p_1\wedge p_2=0$, this gives the statement for $k=2$. The general case follows from this by induction, provided we can show that
$$(p_1\vee p_2\vee \cdots \vee p_i)\wedge p_{i+1}=0$$
for all $i=1,\dots,k-1$.
Denote by $\H_i:=\text{kernel}(P-\lambda_i 1)\subset L^2(\M,\tau)$ the eigenspace corresponding to $\lambda_i$. Then
we have to show that
$$\overline{\H_1+\cdots +\H_i}\cap \H_{i+1}=\{0\}.$$
We do this by induction. For $i=1$, we have for $\xi\in \H_1\cap \H_2$ that 
$$\lambda_1\xi=P\xi=\lambda_2\xi,\qquad\text{thus}\qquad \xi=0.$$
For the induction step consider
$\xi\in\overline{\H_1+\cdots +\H_i}\cap \H_{i+1}$,
thus $\xi\in \H_{i+1}$ and
$$\xi=\lim_{n\to\infty} (\xi_1^{(n)} +\cdots +\xi_i^{(n)}).$$
Thus we have
$$
\lambda_{i+1}\xi=P\xi=\lim_{n\to\infty} (\lambda_1\xi_1^{(n)} +\cdots +\lambda_i\xi_i^{(n)})
=\lambda_1 \xi +\lim_{n\to\infty}[ (\lambda_2-\lambda_1)\xi_2^{(n)}+\cdots +
(\lambda_i-\lambda_1)\xi_i^{(n)}]
$$
This means that the limit
$$\hat\xi:=\lim_{n\to\infty}[ (\lambda_2-\lambda_1)\xi_2^{(n)}+\cdots +
(\lambda_i-\lambda_1)\xi_i^{(n)}]\in \overline{\H_2+\cdots +\H_i}$$
exists and that we have
$$\xi=\frac 1{\lambda_{i+1}-\lambda_1}\hat\xi \in \overline{\H_2+\cdots +\H_i}\cap \H_{i+1}.$$
But the latter intersection consists, by the induction hypothesis, only of $0$.
\end{proof}


\begin{proof}[Proof of Proposition \ref{prop:finite atoms}]
For each $\lambda\in\sigma_\A(P(X))$, we write $p_\lambda:=p_{\ker(P(X)-\lambda\1_N)}$; then we have
\[
\Tr_{N}\circ\tau^{(N)}(p_\lambda) = N - \rank(P(X)-\lambda\1_N) = N - \rho(P-\lambda\1_N)
\]
by combining \eqref{eq:rank_kernel} with Theorem \ref{thm:central eigenvalues and spectrum}. Since for each $\lambda\in\sigma^\full(P)$, $\rho(P)$ is an integer strictly less than $N$, we have
\[
\Tr_{N}\circ\tau^{(N)}(p_\lambda)\geq 1.
\]
Now we consider the finite von Neumann algebra $M_N(M)$ with the trace $\varphi:=\tr_{N}\circ\tau^{(N)}$ instead of the unnormalized one; and then we have $\varphi(p_\lambda)\geq\frac{1}{N}$ for any $\lambda\in\sigma_\A(P(X))$. Suppose that we have distinct ${\lambda_1},\dots,\lambda_k\in\sigma_\A(P(X))$ for some positive integer $k$, then, by the lemma above,
\[
1\geq\varphi(\bigvee_{i=1}^k p_{\lambda_i})=\sum_{i=1}^k\varphi(p_{\lambda_i})\geq \frac{k}{N},
\]
and thus $k\leq N$.
\end{proof}


\subsection{Regularity of matrices with linear entries}\label{subsec:regularity_linear_matrices}

Recently, operator-valued semicircular elements and especially matrix-valued semicircular elements have started to attract much attention, motivated by far reaching applications in random matrix theory. Matrix-valued semicircular elements are noncommutative random variables of the form
$$\bS = b_0 + b_1 S_1 + \dots + b_n S_n$$
with selfadjoint coefficient matrices $b_0,b_1,\dots,b_n \in M_N(\C)$ and a tuple $(S_1,\dots,S_n)$ of freely independent semicircular elements. In some impressive series of publications (see, for instance, \cite{EKYY13,AjEK18,AEK18} and the references collected therein), a deep understanding of the regularity properties of their distributions was gained. These achievements rely on a very detailed analysis of the so-called \emph{Dyson equation}, which is some kind of quadratic equation on $M_N(\C)$ (or, more generally, on von Neumann algebras) for their operator-valued Cauchy transforms that determines in particular their scalar-valued Cauchy transforms and hence their analytic distributions $\mu_{\bS}$.

It turns out that analytic properties of $\mu_{\bS}$ strongly depend on the algebraic properties of the coefficient matrices $b_0,b_1,\dots,b_n$. In \cite{AjEK18}, the associated \emph{self-energy operator} (or \emph{quantum operator} in the terminology of \cite{GGOW16} used in Proposition \ref{prop:rank-decreasing})
$$\cL:\ M_N(\C) \to M_N(\C),\qquad b \mapsto \sum^n_{j=1} b_j b b_j$$
was supposed to be \emph{flat} in the sense that there are constants $c_1,c_2>0$ such that
\begin{equation}\label{eq:flatness}
c_1 \tr_N(b) \1_N \leq \cL(b) \leq c_2 \tr_N(b) \1_N \qquad\text{for all positive semidefinite $b\in M_N(\C)$}. 
\end{equation}
Note that $b_0$, the so-called \emph{bare matrix}, plays a special role and accordingly does not show up in $\cL$. 

It is natural to ask which of the results that were obtained for matrix-valued semicircular elements survive if the Dyson equation is replaced by some other, more general assumption on $(X_1,\dots,X_n)$ and how much one can weaken the rather strong flatness condition. Indeed, the results that we obtained above allow us to deal with more general matrix-valued elements that are of the form
$$\X = b_0 + b_1 X_1 + \dots + b_n X_n,$$
where, on the analytic side, we allow $(X_1,\dots,X_n)$ to be any tuple of (not necessarily selfadjoint) noncommutative random variables satisfying the condition $\Delta(X)=n$; on the algebraic side, we significantly relax the flatness condition by requiring only fullness of the associated linear matrix
$$P = b_0 + b_1 x_1 + \dots + b_n x_n \in M_N(\C\langle x_1,\dots,x_n\rangle).$$

\begin{remark}
Let $P\in M_N(\C\langle x_1,\dots,x_n\rangle)$ be linear, say $P = b_0 + b_1 x_1 + \dots + b_n x_n$. We note that if the associated self-energy operator $\cL$ is supposed to satisfy the lower estimate in \eqref{eq:flatness}, then $P$ must full; in particular, this shows that flatness of $\cL$ is indeed a much stronger requirement than fullness of $P$.
To see this, we observe first that the lower estimate in \eqref{eq:flatness} enforces $\cL$ to be nowhere rank-decreasing, so that, according to Proposition \ref{prop:rank-decreasing}, the homogeneous part $b_1 x_1 + \dots + b_n x_n$ of $P$ and hence (as one sees for instance with the help of Proposition \ref{prop:spectrum of linear matrices}) also $P$ must be full.
\end{remark}

The following theorem summarizes the regularity results for such matrix-valued elements; it follows immediately from Proposition \ref{prop:spectrum of linear matrices} and Theorem \ref{thm:central eigenvalues and spectrum}.

\begin{theorem}
Suppose that $X=(X_1,\dots,X_n)$ is a tuple of (not necessarily selfadjoint) noncommutative random variables in some tracial $W^\ast$-probability space $(\M,\tau)$ that satisfies $\Delta(X)=n$. Consider any linear matrix $P = b_0 + b_1 x_1 + \dots + b_n x_n$ in $M_N(\C\langle x_1,\dots,x_n\rangle)$ for which the noncommutative random variable
\[
\X = P(X) = b_0 + b_1 X_1 + \dots + b_n X_n
\]
in the tracial $W^\ast$-probability space $(M_N(\M),\tr_N \circ \tau^{(N)})$ is normal. Then the analytic distribution $\mu_\X$ of $\X$ has atoms precisely at the points in $\sigma^\full(P)$; in fact, we have that
\[
\mu_\X(\{\lambda\}) = \frac{1}{N}\big(N - \rho(P - \lambda \1_N)\big) \qquad \text{for each $\lambda\in\C$}.
\]
Moreover, the following statements hold true:
\begin{enumerate}
 \item We have that $\sigma_p(\X) \subseteq \sigma(b_0)$, i.e., $\mu_\X$ can have atoms only at eigenvalues of the constant matrix $b_0$.
 \item If the homogeneous part $b_1 x_1 + \cdots + b_n x_n$ of $P$ is full, then $\sigma_p(\X)=\emptyset$, i.e., $\mu_\X$ has no atoms.
\end{enumerate}
\end{theorem}

\subsection{Strong Atiyah property}\label{subsec:Atiyah}

In Item \ref{it:Atiyah-1_iii} of Theorem \ref{thm:Atiyah-1}, the equality $\rank(P(X))=\rho(P)$ means in particular that $\rank(P(X))$, which a priori can be any real number between $0$ and $N$, has actually to be an integer, since the algebraic rank $\rho$ is by definition an integer. The question whether the analytic rank of $P(X)$ for an operator tuple $X=(X_1,\dots,X_n)$ is for all polynomial matrices $P$ an integer has received quite some attention and is related with the famous Atiyah conjecture. Whereas the original Atiyah conjecture is in the context of generators of group algebras, Shlyakhtenko and Skoufranis extended this in \cite{SS15} to the general setting of operator tuples with the following definition.

\begin{definition}\label{def:strong_Atiyah_property}
Let $X=(X_{1},\dots,X_{n})$ be a tuple with elements from a tracial $W^{\ast}$-probability space $(\M,\tau)$, and consider the evaluation map $\ev_{X}:\C\left\langle x_{1},\dots,x_{n}\right\rangle \rightarrow\M$. If for any matrix $P\in M_{N}(\C\left\langle x_{1},\dots,x_{n}\right\rangle )$, we have
\[\rank(P(X))\in\N,\]
then we say $X$ has the \emph{strong Atiyah property}.
\end{definition}

The presence of this property is one of the various formulations of the Atiyah conjecture, which arose in the work \cite{Ati74} and asks whether some analytic $L^{2}$-Betti numbers are always rational numbers for certain Riemannian manifolds. In the terminology of Definition \ref{def:strong_Atiyah_property}, Linnell proved in \cite{Lin93} that the generators of the free group (considered as operators in the left regular representation) have the strong Atiyah property. See also \cite[Chapter 10]{Luc02} or \cite{GLSZ00,DLMSY03,PTh11} for more detailed discussions, including some counterexamples, on the Atiyah conjecture for group algebras and related subjects. In the context of free probability, Shlyakhtenko and Skoufranis proved in \cite{SS15} that a tuple of non-atomic, freely independent random variables has the strong Atiyah property. A consequence of our Corollary \ref{cor:Delta_maximality_implies_Atiyah} is the following vast generalization of their result.

\begin{theorem}\label{cor:full_entropy_dimension_implies_Atiyah2}
If $X=(X_{1},\dots,X_{n})$ is a tuple of random variables in some tracial $W^{\ast}$-probability space $(\M,\tau)$ with $\Delta(X)=n$, then $X$ has the strong Atiyah property.
\end{theorem}

Note that Item \ref{it:Atiyah-1_iii} of Theorem \ref{thm:Atiyah-1} is actually stronger than the strong Atiyah property: the strong Atiyah property only asks the rank of matrices to be integers, but in Theorem \ref{thm:Atiyah-1}, we ask the rank to be exactly the corresponding inner rank (which is an integer by definition).

As explained in \cite{SS15}, the strong Atiyah property implies that whenever we take a $P\in M_N(\C\langle x_1,\dots,x_n\rangle)$ for which the associated operator $\X=P(X_1,\dots,X_n)$ in $M_N(\M)$ is normal, then the measure of any atom of the spectral distribution of $\X$ is necessarily an integer multiple of $\frac{1}{N}$; that more can be said if the third property in Theorem \ref{thm:Atiyah-1} holds, has been explained in Theorem \ref{thm:central eigenvalues and spectrum} and Remark \ref{rem:atoms}.

In general, however, it is possible that the rank is an integer but is not equal to the inner rank. 
Let us present an example, provided by Ken Dykema and James Pascoe, to show this gap between the strong Atiyah property and the properties in Theorem \ref{thm:Atiyah-1}.

\begin{example}\label{ex:Dykema_Pascoe}
Consider two freely independent semicircular elements, denoted by $X$ and $Y$. By the results from \cite{SS15} (or by our Theorem \ref{cor:full_entropy_dimension_implies_Atiyah2}), they satisfy the strong Atiyah property. Let
\[A=Y^{2},\ B=YXY,\ C=YX^{2}Y,\]
then $A$, $B$, $C$ also have the strong Atiyah property as any polynomial in them can be reduced back to a polynomial in $X$ and $Y$. However, though they do not satisfy any nontrivial polynomial relation, they have a rational relation:
\[BA^{-1}B-C=0\]
in the $\ast$-algebra of affiliated operators. Then definitely they do not satisfy the last property in Theorem \ref{thm:Atiyah-1}; moreover, we can also find some matrix like
\[\begin{pmatrix}A & B\\B & C\end{pmatrix}\]
that has inner rank $2$ if it is viewed as a matrix of three formal variables, but has
\[\rank\begin{pmatrix}A & B\\B & C\end{pmatrix}=\rank\begin{pmatrix}Y^{2} & YXY\\YXY & YX^{2}Y\end{pmatrix}=\rank\begin{pmatrix}1 & X\\X & X^{2}\end{pmatrix}=\rho\begin{pmatrix}1 & x\\x & x^{2}\end{pmatrix}=1.\]
Therefore, $(A,B,C)$ violates all the properties in Theorem \ref{thm:Atiyah-1} though it has the strong Atiyah property.
\end{example}

Nevertheless, by the following list of equivalent properties, we see that the reason for the strong Atiyah property to hold is always the fact that the analytic rank is equal to an inner rank, however over the rational closure (which in general, does not need to be the free field).

\begin{theorem}
\label{thm:Atiyah-2}
Let $(\M,\tau)$ be a tracial $W^{\ast}$-probability space and $\A$ the $\ast$-algebra of affiliated unbounded operators.
For a given tuple $X=(X_{1},\dots,X_{n})$ in $\M^n$, we consider the evaluation map $\ev_{X}:\C\left\langle x_{1},\dots,x_{n}\right\rangle \rightarrow\A$. Let
$\mathcal{R}$ be the rational closure of $\C\left\langle x_{1},\dots,x_{n}\right\rangle $ with respect to this evaluation map. We denote the inner rank over $\mathcal{R}$ by $\rho_{\mathcal{R}}$. Then the following statements are equivalent.
\begin{enumerate}
\item For any $N\in\N$ and any $P\in M_{N}(\C\left\langle x_{1},\dots,x_{n}\right\rangle )$ we have: if $P(X)$ is full over $\mathcal{R}$, then $P(X)\in M_{N}(\A)$ is invertible.
\item For any $N\in\N$ and any $P\in M_{N}(\C\left\langle x_{1},\dots,x_{n}\right\rangle )$ we have: $\rank(P(X))=\rho_{\mathcal{R}}(P(X))$.
\item The rational closure $\mathcal{R}$ of $\C\left\langle X_{1},\dots,X_{n}\right\rangle $ is a division ring.
\item The division closure $\mathcal{D}$ of $\C\left\langle X_{1},\dots,X_{n}\right\rangle $ is a division ring.
\item $X$ has the strong Atiyah property, i.e., for any $N\in\N$ and any $P\in M_{N}(\C\left\langle x_{1},\dots,x_{n}\right\rangle )$ we have that $\rank(P(X))\in\N$.
\end{enumerate}
\end{theorem}

\begin{proof}
It's easy to see that (ii)$\implies$(i) follows from Lemma \ref{invertible and rank}; and that the equivalence between (iii) and (iv) is Proposition \ref{prop:division closure}.

(iii)$\implies$(ii): Assume that the evaluation of $P\in M_{N}(\C\left\langle x_{1},\dots,x_{n}\right\rangle)$ has inner rank $\rho_{\mathcal{R}}(P(X))=r$. By Lemma \ref{lem:diagonalization}, there exist two invertible matrices $U$ and $V$ over the division ring $\mathcal{R}$ such that
\[
UP(X)V=\begin{pmatrix}\1_r & \0\\\0 & \0\end{pmatrix}.
\]
Then by Lemma \ref{inverstible matrix preserve rank} we have
\[
\rank(P(X))=\rank(UP(X)V)=\rank\begin{pmatrix}\1_{r} & \0\\\0 & \0\end{pmatrix}=r.
\]

(i)$\implies$(iii): For any nonzero $r\in\mathcal{R}$ , there exist a matrix $P\in M_{N}(\C\left\langle x_{1},\dots,x_{n}\right\rangle )$, $u\in M_{1,N}(\C)$, $v\in M_{N,1}(\C)$ such that $P(X)$ is invertible in $M_{N}(\A$) and $r=uP(X)^{-1}v$ (see Section \ref{subsec:rational closure}). Then Schur's Lemma (Lemma \ref{lem:Schur complement}) tells us that the matrix
\[
A:=\begin{pmatrix}0 & u\\v & P(X)\end{pmatrix}
\]
is invertible if and only if its Schur complement, $-uP(X)^{-1}v$, is invertible. Therefore, in order to show that $r=uP(X)^{-1}$ is invertible, it suffices to show that the matrix $A$ is full over $\mathcal{R}$ by the assumption (i). For that purpose, we can apply Proposition \ref{prop:invertible minor} as $\mathcal{R}$ is stably finite (because it is a subalgebra of $\A$ which is stably finite), then we see that $A$ is full over $\mathcal{R}$ since $-uP(X)^{-1}v=-r\neq0$. Moreover, because the rational closure $\mathcal{R}$ is closed under taking inverses, $\mathcal{R}$ becomes a division ring.

Finally, it remains to prove that the assertion (v) is equivalent to the first four assertions. It is clear that (ii) implies (v) trivially, as the inner rank is always an integer by definition. To complete the proof, we want to prove (iii) from assertion (v). Given any nonzero $r\in\mathcal{R}$, there exists a matrix $P\in M_{N}(\C\left\langle x_{1},\dots,x_{n}\right\rangle )$, $u\in M_{1,N}(\C)$, $v\in M_{N,1}(\C)$ such that $P(X)$ is invertible in $M_{N}(\A$) and $r=uP(X)^{-1}v$. Considering the factorization
\[\begin{pmatrix}-r & \0\\\0 & \1_{N}\end{pmatrix}=\begin{pmatrix}1 & -uP(X)^{-1}\\\0 & P(X)^{-1}\end{pmatrix}\begin{pmatrix}0 & u\\v & P(X)\end{pmatrix}\begin{pmatrix}1 &\0\\-P(X)^{-1}v & \1_{N}\end{pmatrix},\]
by Lemma \ref{inverstible matrix preserve rank} we have
\[\rank\begin{pmatrix}-r & \0\\\0 & \1_{N}\end{pmatrix}=\rank\begin{pmatrix}0 & u\\v & P(X)\end{pmatrix}.\]
Then by assertion (v), we know
\[
\rank\begin{pmatrix}0 & u\\v & P(X)\end{pmatrix}\in\N;
\]
combining this with the fact that
\[\rank\begin{pmatrix}-r & \0\\\0 & \1_{N}\end{pmatrix}=\rank r+N,\]
we obtain $\rank(r)\in\{0,1\}$. Then, as $r\neq0$, we have $\rank(r)=1$, and thus $r$ is invertible by Lemma \ref{invertible and rank}. Hence $\mathcal{R}$ is a division ring as $\mathcal{R}$ is closed under taking inverses.
\end{proof}

\begin{remark}
Theorem \ref{thm:Atiyah-2} says in particular that the strong Atiyah property is also equivalent to the property that the division closure is a division ring. Actually, this equivalence is known for the group case; see \cite[Lemma 10.39]{Luc02}.
\end{remark}

\begin{remark}
We can choose the tuple $X$ in Theorem \ref{thm:Atiyah-2} to be a tuple of commuting operators; in particular, $X$ can be chosen to be a tuple of classical random variables. In such cases, the evaluation map $\ev_X$ can be regarded as the evaluation map on the algebra of commutative polynomials $\C\left[x_1,\dots,x_n\right]$, and then it can also be extended (as large as possible) to the field of fractions $\C\left(x_1,\dots,x_n\right)$. Therefore, Theorem \ref{thm:Atiyah-2} can be stated for commuting operators where we simply replace the noncommutative polynomials by commutative ones and consider the corresponding rational closure and division closure.
\end{remark}

\section{Maximality of $\Delta$}

Our main result, Theorem \ref{thm:intro-main}, gives a characterization of various properties of a tuple of operators $X=(X_1,\dots,X_n)$ in terms of the non-commutative distribution of  $X$. However, this characterization is in terms of the quantity $\Delta$, which is not so easy to calculate directly. Thus it is important to have other, more accessible, criteria to decide upon whether $\Delta$ is maximal. One important quantity in this context is the free entropy dimension $\delta^*$; it turns out that maximality of this implies also maximality of $\Delta$; and we have a lot of free probability tools for deciding upon maximality of $\delta^*$. Actually, in a first version of this paper, our results were formulated under the assumption of maximality of $\delta^*$, before we noticed that we can get even an equivalence if we use $\Delta$ instead. One drawback of $\delta^*$ is that this is only defined for selfadjoint variables; and hence it does not allow us to address the question whether the unitary generators of the free group generate the free field. Since we also wanted to include Linnell's positive answer to this in our case, we took actually quite some efforts to formulate and prove our main theorems for general, not necessarily selfadjoint, tuples. But in order to apply this then to the free group case we need another way of checking that $\Delta$ is maximal in this case. It turns out that another important conceptual tool in free probability, so-called dual systems, allows to do this. Usually dual systems are also only considered for selfadjoint operators, but in contrast to the free entropy dimension situation, an extension to non-selfadjoint operators is here not a problem.

\subsection{Free entropy dimension}\label{sect:free-entropy-dimension}

Voiculescu introduced non-microstates free versions of Fisher information, $\Phi^*(X_1,\dots,X_n)$, and entropy, $\chi^*(X_1,\dots,X_n)$, for tuples of selfadjoint operators. For a precise definition of those and their properties we refer to \cite{Voi98}.

The \emph{non-microstates free entropy dimension} $\delta^\ast(X_1,\dots,X_n)$ is defined in terms of the non-microstates free entropy $\chi^\ast$ by
$$\delta^\ast(X_1,\dots,X_n) := n - \liminf_{\epsilon \searrow 0} \frac{\chi^\ast(X_1+\sqrt{\epsilon}S_1,\dots,X_n+\sqrt{\epsilon}S_n)}{\log(\sqrt{\epsilon})}.$$
In the case $n=1$ of a single operator $X=X^\ast\in\M$ one has that
\begin{equation}\label{eq:entropy_dimension_single_operator}
\delta^\ast(X) = 1 - \sum_{t\in\R} \mu_X(\{t\})^2,
\end{equation}
where $\mu_X$ is the analytic distribution of $X$. In particular, $\delta^\ast(X)=1$ is then exactly the statement that $\mu_X$ has no atoms.

We note that in \cite{CS05} some variant of $\delta^\ast(X_1,\dots,X_n)$ was introduced, namely
$$\delta^\star(X_1,\dots,X_n) := n - \liminf_{t\searrow0} t \Phi^\ast(X_1+\sqrt{t}S_1,\dots,X_n+\sqrt{t}S_n),$$
whose defining expression is formally obtained by applying L'Hospital's rule to the $\liminf$ appearing in the definition of $\delta^\ast(X_1,\dots,X_n)$. 

It was shown in \cite{CS05} that
\[
\delta^\ast(X_1,\dots,X_n) \leq \delta^\star(X_1,\dots,X_n)\leq\Delta(X_1,\dots,X_n)\leq n
\]
holds. So the condition $\Delta(X_1,\dots,X_n)=n$  sits at the end of the following chain of implications
\begin{align*}
 \Phi^\ast(X)<\infty & \implies \chi^\ast(X_1,\dots,X_n)>-\infty\\
  & \implies \delta^\ast(X_1,\dots,X_n)=n\\
  & \implies \delta^\star(X_1,\dots,X_n)=n\\
  & \implies \Delta(X_1,\dots,X_n)=n.
\end{align*}

This means in particular that under the assumption $ \delta^\star(X_1,\dots,X_n)=n$ we get all the implications which we have proved under the assumption $\Delta(X_1,\dots,X_n)=n$. 
This shows in particular that our results here 
generalize the result in \cite{MSW17,CS16}, where the absence of zero divisors was shown for polynomials in operator tuples with maximal free entropy dimension $\delta^\star$. 

It is of course important to note that we have many tools for deciding upon the maximality of the free entropy dimension; in particular, this is given for free variables if each of them has non-atomic distributions. 

\subsection{Dual systems} Dual systems were introduced by Voiculescu in \cite{Voi98}, and appeared also in \cite{CS05} as an important technical tool for getting the maximality of the free entropy dimension. There only the selfadjoint situation was considered. We adapt here those ideas to the general case.

\begin{definition}
Let $X_1,\dots,X_n$ be a operators in a tracial $W^*$-probability space $(\M,\tau)$. A \emph{dual system} to $X_1,\dots,X_n$ are operators $D_1,\dots,D_n\in B(L^2(\M,\tau))$ such that we have
$$[X_i,D_j]=\delta_{ij}P_\Omega,$$
where $P_\Omega\in B(L^2(\M,\tau))$ is the orthogonal projection onto the trace vector $\Omega$ in $L^2(\M,\tau)$.
\end{definition}

\begin{proposition}
Let $(X_1,\dots,X_n)$ be a tuple of operators in a tracial $W^*$-probability space $(\M,\tau)$. If $X_1,\dots,X_n$ has a dual system, then
$\Delta(X_1,\dots,X_n)=n$.
\end{proposition}

\begin{proof}
Let $D_1,\dots,D_n$ be a dual system for $X_1,\dots,X_n$.
Assume that $T_1,\dots,T_n$ are finite rank operators on $L^2(\M,\tau)$ satisfying
$$\sum_{i=1}^n [T_i,X_i]=0.$$
Let $Q_1$ and $Q_2$ be in the commutant of $\M$ on $L^2(\M,\tau)$. Then we have
\begin{align*}
0&=\Tr\left(\sum_i [T_i,X_i] Q_1D_jQ_2\right)\\
&=\sum_i\Tr\left(T_iX_iQ_1D_jQ_2-X_iT_iQ_1D_jQ_2\right)\\
&=\sum_i\Tr\left(T_iQ_1(X_iD_j-D_jX_i)Q_2\right)\\
&=\Tr(T_jQ_1P_\Omega Q_2)\\
&=\langle Q_2T_jQ_1 \Omega,\Omega\rangle\\
&=\langle T_j Q_1\Omega, Q_2^* \Omega\rangle.
\end{align*}
Since $\Omega$ is cyclic for the commutant, we can get with $Q_1\Omega$ and $Q_2^*\Omega$ any vector from a dense subset of $L^2(\M,\tau)$. But then the above implies that $T_j=0$, for each $j$.
\end{proof}

We want to apply this now to the generators of the free group. Thus we have to present a dual system for them. The idea is that in the case of one variable, the dual operator for the two-sided shift is given by the one-sided shift, and then one can lift this via taking free products to the $n$ variable case, by imitating Proposition 5.6 from \cite{Voi98}.

\begin{proposition}
Let $U_1=\lambda(g_1),\dots,U_n=\lambda(g_n)$ be the generators of the free group $\FF_n$ in the left regular representation on $L^2(\FF_n)$. For each $i$, let $\FF_n^{(i)}$ be those elements from the free group which end in their reduced form with $g_i$, and let $\FF_n^{[i]}$ be its complement. 
Then we define the operator $V_i$ as follows: restriced to $L^2(\FF_n^{(i)})$ it acts by right multiplication with $g_i^{-1}$, and restricted to $L^2(\FF_n^{[i]})$ it is zero.
Then $V_1,\dots,V_n$ is a dual system for $U_1,\dots,U_n$.
\end{proposition}

\begin{proof}
It suffices to check the commutator relation $[U_i,V_j]=\delta_{ij}P$ (where $P$ is the projection onto the neutral element) for vectors which correspond to group elements $h\in \FF_n$. Let $h=e$ be the neutral element. Then
$$(U_iV_j-V_jU_i)e=0-V_jg_i=\delta_{ij} e.$$
Consider now an element $h=h'g_j$ which ends with a $g_j$. Note that $h'$ cannot end with a $g_j^{-1}$. Then we have
$$(U_iV_j-V_jU_i)h'g_j=U_i h'-V_jg_ih'g_j=g_ih'-g_ih'=0.
$$
Now consider an element $h=h'g$ which ends with a generator or its inverse, with $g\not= g_j$. Then
$$(U_iV_j-V_jU_i)h'g=0-V_j g_ih'g=0
$$

\end{proof}

This gives then as a corollary the result of Linnell \cite{Lin93} for the free group.

\begin{corollary}
Let $U_1,\dots,U_n$ be the 
the generators of the free group $\FF_n$ in the left regular representation on $L^2(\FF_n)$. Then $\Delta(U_1,\dots,U_n)=n$, and hence the division (and the rational) closure of the polynomials in $U_1,\dots,U_n$ in the unbounded operators affiliated to $L(\FF_n)$ is the free field.
\end{corollary}

\bibliographystyle{amsalpha}
\bibliography{rational_regularity}

\end{document}